\documentclass[]{imsart}

\usepackage{amsthm,amsmath,amssymb}

\usepackage{graphicx}
\usepackage{tcolorbox}
\usepackage{mathtools}

\newtheorem{condition}{Condition}
\newtheorem{theorem}{Theorem}
\newtheorem{lemma}{Lemma}
\newtheorem{remark}{Remark}
\newtheorem{corollary}{Corollary}

\startlocaldefs
\usepackage{xcolor}
\newcommand{\Csup}{C_{\mathrm{sup}}}
\newcommand{\Cinf}{C_{\mathrm{inf}}}
\input{standardmacros.sty}
\endlocaldefs

\begin{document}
\graphicspath{{./figures/}}
\DeclareGraphicsExtensions{.pdf}

\begin{frontmatter}

\title{Asymptotic properties of the maximum likelihood and cross validation estimators for transformed Gaussian processes}

\runtitle{Asymptotics for transformed Gaussian processes}

\begin{aug}

\author{\fnms{Francois} \snm{Bachoc}\thanksref{a}\corref{}\ead[label=e1]{francois.bachoc@math.univ-toulouse.fr}},
\author{\fnms{Jos\'e} \snm{B\'etancourt}\thanksref{b}\ead[label=e2]{jose.betancourt@math.univ-toulouse.fr}},
\author{\fnms{Reinhard} \snm{Furrer}\thanksref{c}\ead[label=e3]{reinhard.furrer@math.uzh.ch}}
\and
\author{\fnms{Thierry} \snm{Klein}\thanksref{d}\ead[label=e4]{thierry.klein@math.univ-toulouse.fr}}
\address[a]{Institut de math\'ematique, UMR5219; Universit\'e de Toulouse;
CNRS, UPS IMT, F-31062 Toulouse Cedex 9, France; \printead{e1}}
\address[b]{Institut de math\'ematique, UMR5219; Universit\'e de Toulouse;
CNRS, UPS IMT, F-31062 Toulouse Cedex 9, France; \\
ENAC - Ecole Nationale de l'Aviation Civile, Universit\'e de
Toulouse, France; \printead{e2}}
\address[c]{Department of Mathematics and Department of Computational Science, University of Zurich, CH-8057 Zurich; \printead{e3}}
\address[d]{ENAC - Ecole Nationale de l'Aviation Civile, Universit\'e de
Toulouse, France;  \\
Institut de math\'ematique, UMR5219; Universit\'e de Toulouse;  \printead{e4}}
\runauthor{Bachoc, B\'etancourt, Furrer, Klein}

\affiliation{University Paul Sabatier and University of Zurich}

\end{aug}

\begin{abstract}
The asymptotic analysis of covariance parameter estimation of Gaussian processes has been subject to intensive investigation. However, this asymptotic analysis is very scarce for non-Gaussian processes. In this paper, we study a class of non-Gaussian processes obtained by regular non-linear transformations of Gaussian processes.
We provide the increasing-domain asymptotic properties of the (Gaussian) maximum likelihood and cross validation estimators of the covariance parameters of a non-Gaussian process of this class. We show that these estimators are consistent and asymptotically normal, although they are defined as if the process was Gaussian. They do not need to model or estimate the non-linear transformation. Our results can thus be interpreted as a robustness of (Gaussian) maximum likelihood and cross validation towards non-Gaussianity. Our proofs rely on two technical results that are of independent interest for the increasing-domain asymptotic literature of spatial processes. First, we show that, under mild assumptions, coefficients of inverses of large covariance matrices decay at an inverse polynomial rate as a function of the corresponding observation location distances. Second, we provide a general central limit theorem for quadratic forms obtained from transformed Gaussian processes. Finally, our asymptotic results are illustrated by numerical simulations.
\end{abstract}



\end{frontmatter}

\section{Introduction}

Kriging \cite{stein99interpolation,rasmussen06gaussian} consists of inferring the values of a (Gaussian) random field given observations at a finite set of points.
It has become a popular method for a
large range of applications, such as geostatistics \cite{matheron70theorie}, numerical code approximation \cite{sacks89design,santner03design,bachoc16improvement}, calibration \cite{paulo12calibration,bachoc14calibration}, global optimization \cite{jones98efficient}, and machine learning \cite{rasmussen06gaussian}.

When considering a Gaussian process, one has to deal with the estimation of its covariance function.
Usually, it is assumed that the covariance function belongs to a given parametric family (see \cite{abrahamsen97review} for a review of classical families).
In this case, the estimation boils down to estimating the corresponding covariance parameters. Nowadays, the main estimation techniques are based on maximum likelihood \cite{stein99interpolation,rasmussen06gaussian}, cross-validation \cite{zhang10kriging,bachoc2013cross,bachoc14asymptotic,bachoclagnoux2017} and variation estimators \cite{istas97quadratic,And2010,ABKLN18}. 

The asymptotic properties of estimators of the covariance parameters have been widely studied in the two following frameworks. The fixed-domain asymptotic framework, sometimes called infill asymptotics \cite{stein99interpolation, Cre1993}, corresponds to the case where more and more data are observed in some fixed bounded sampling domain. The increasing-domain asymptotic framework corresponds to the case where the sampling domain  increases with the number of observed data. 

Under fixed-domain asymptotics, and particularly in low dimensional settings, not all covariance parameters can be estimated consistently (see \cite{ibragimov78gaussian,stein99interpolation}). Hence, the distinction is made between microergodic and non-microergodic covariance parameters \cite{ibragimov78gaussian,stein99interpolation}.  Although non-micro\-ergodic parameters cannot be estimated consistently, they have an asymptotically negligible impact on prediction \cite{AEPRFMCF,BELPUICF,UAOLPRFUISOS,zhang04inconsistent}.
There is, however, a fair amount of literature on the consistent estimation of microergodic parameters (see for instance \cite{zhang04inconsistent,ShaKau2013,DuZhaMan2009,WanLoh2011,ying91asymptotic,ying93maximum}).  

This paper focuses on the increasing-domain asymptotic framework.
Indeed, generally speaking, increasing-domain asymptotic results hold for significantly more general families of covariance functions than fixed-domain ones.
Under increasing-domain asymptotics, the maximum likelihood and cross validation estimators of the covariance parameters are consistent and asymptotically normal under mild regularity conditions \cite{mardia84maximum,shaby12tapered,bachoc14asymptotic,furrer2016asymptotic}.

All the asymptotic results discussed above are based on the assumption that the data come from a Gaussian random field. This assumption is indeed theoretically convenient but might be unrealistic for real applications. When the data stem from a non-Gaussian random field, it is still relevant to estimate the covariance function of this random field. Hence, it would be valuable to extend the asymptotic results discussed above to the problem of estimating the covariance parameters of a non-Gaussian random field.

In this paper, we provide such an extension, in the special case where the non-Gaussian random field is a deterministic (unknown) transformation of a Gaussian random field. Models of transformed Gaussian random fields have been used extensively in practice
(for example in \cite{chiles2009geostatistics,xu2017tukey,alegria2017estimating,yan2018gaussian}).

Under reasonable regularity assumptions, we prove that applying the (Gaussian) maximum likelihood estimator to data from a transformed Gaussian random field yields a consistent and asymptotically normal estimator of the covariance parameters of the transformed random field. This (Gaussian) maximum likelihood estimator corresponds to what would typically be done in practice when applying a Gaussian process model to a non-Gaussian spatial process. This estimator does not need to know the existence of the non-linear transformation function and is not based on the exact density of the non-Gaussian data. We refer to Remark~\ref{remark:explaination:ML} for further details and discussion on this point.

We then obtain the same consistency and asymptotic normality result when considering a cross validation estimator. In addition, we establish the joint asymptotic normality of both these estimators, which provides the asymptotic distribution of a large family of aggregated estimators. Our asymptotic results on maximum likelihood and cross validation are illustrated by numerical simulations.

To the best of our knowledge, our results (Theorems~\ref{theorem:consistency:thetaML}, \ref{theorem:CLT:thetaML}, \ref{theorem:consistency:thetaCV}, \ref{theorem:CLT:thetaCV} and \ref{theorem:CLT:joint}) provide the first increasing-domain asymptotic analysis of Gaussian maximum likelihood and cross validation for non-Gaussian random fields. Our proofs intensively rely on Theorems~\ref{theorem:decay:coeff:Rinverse} and~\ref{theorem:generic:TCL}. Theorem~\ref{theorem:decay:coeff:Rinverse} shows that the components of inverse covariance matrices are bounded by inverse polynomial functions of the corresponding distance between observation locations. Theorem~\ref{theorem:generic:TCL} provides a generic central limit theorem for quadratic forms constructed from transformed Gaussian processes. These two theorems have an interest in themselves. 

The rest of the paper is organized as follows. In Section~\ref{section:general properties}, general properties of transformed Gaussian processes are provided. In Section~\ref{section:two:main:technical:results},  Theorems~\ref{theorem:decay:coeff:Rinverse} and~\ref{theorem:generic:TCL} are stated. In Section~\ref{section:sigma}, an application of these two theorems is given to the case of estimating a single variance parameter. In Section~\ref{sec:estitheta}, the consistency and asymptotic normality results for general covariance parameters are given. The joint asymptotic normality result is also given in this section. The simulation results are provided in Section~\ref{section:simulation:result}. All the proofs are provided in the appendix.

\bigskip

\section{General properties of transformed Gaussian processes} \label{section:general properties}

In applications, the use of Gaussian process models may be too restrictive. One possibility for obtaining larger and more flexible classes of random fields is to consider transformations of Gaussian processes.
In this section, we now introduce the family of transformed Gaussian processes that we will study asymptotically in this paper. This family is determined by regularity conditions on the covariance function of the original Gaussian process and on the transformation function. 

Let us first introduce some notation. Throughout the paper, $\Cinf > 0$ (resp. $\Csup < \infty$) denotes a generic strictly positive (resp. finite) constant. This constant never depends on the number of observations $n$, or on the covariance parameters (see Section~\ref{sec:estitheta}), but is allowed to depend on other variables. We mention these dependences explicitly in cases of ambiguity. 
 The values of $\Cinf$  and $\Csup$ may change across different occurrences.
 
For a vector $x$ of dimension $d$ we let $|x| = \max_{i=1,\dots,d} |x_i|$. Further, the Euclidean and operator norms are denoted by $||x||$ and by  $||  M ||_{op}=\sup\{ ||Mx||: ||x||\leq 1 \}$, for any matrix $M$.
We let $\lambda_1(B) \geq \ldots \geq \lambda_r(B) $ be the $r$ eigenvalues of a $r \times r$ symmetric matrix $B$. We let $\rho_1(B) \geq \ldots \geq \rho_r(B) \geq 0$ be the $r$ singular values of a $r \times r$ matrix $B$.
We let $\mathbb{N}$ be the set of non-zero natural numbers.

Further, we define the Fourier transform of a function $h : \IR^d \to \IR$ by $\hat{h}(f) = $ $ (2 \pi)^{-d} \int_{\IR^d} h(t) e^{- \mathrm{i} f^\top t} dt$, where $\mathrm{i}^2 = -1$.

For a sequence of observation locations, the next condition ensures that a fixed distance between any two observation locations exists. This condition is classical \cite{bachoc14asymptotic,bachoc2016smallest}.

\begin{condition} \label{cond:delta}
We say that a sequence of observation locations, $(x_i)_{i \in \mathbb{N}}, x_i \in \mathbb{R}^d,$ is asymptotically well-separated if we have $ \inf_{i , j \in \mathbb{N}, i \neq j} |x_i - x_j| >0 $.
\end{condition}

The next condition on a stationary covariance function is classical under increasing-domain asymptotics. This condition provides asymptotic decorrelation for pairs of distant observation locations and implies that covariance matrices are asymptotically well-conditioned when a minimal distance between any two distinct observation locations exists \cite{mardia84maximum,bachoc14asymptotic}.

\begin{condition} \label{cond:sub:exp:ass:pos}
We say that a stationary covariance function $k$ on $\mathbb{R}^d$ is sub-exponential and asymptotically positive if:
\begin{itemize}
\item[i)] $\Csup$ and $\Cinf$ exist such that, for all $s \in \mathbb{R}^d$, we have
\begin{equation} \label{eq:k:decay}
|k(s)| \leq 
\Csup \exp{ (- \Cinf |s|) };
\end{equation}
\item[ii)] For any sequence $(x_i)_{i \in \mathbb{N}}$ satisfying Condition~\ref{cond:delta}, we have $\inf_{n \in \mathbb{N}} \lambda_{n}(\Sigma) >0$, where $\Sigma$ is the $n \times n$ matrix $(k(x_i-x_j))_{i,j=1,\dots,n}$.
\end{itemize}

\end{condition}

In Condition~\ref{cond:sub:exp:ass:pos}, we remark that $k : \mathbb{R}^d \to \mathbb{R}$ is called a stationary covariance function in the sense that $(x_1,x_2) \to k(x_1-x_2)$ is a covariance function. We use this slight language abuse for convenience.

We also remark that, when non-transformed Gaussian processes are considered, a polynomial decay of the covariance function in Condition~\ref{cond:sub:exp:ass:pos} i) is sufficient to obtain asymptotic results \cite{bachoc14asymptotic,bachoc2018asymptotic}. Here an exponential decay is needed in the proofs to deal with the non-Gaussian case. Nevertheless, most classical families of covariance functions satisfy inequality \eqref{eq:k:decay}.  

When considering a transformed Gaussian process, we will consider a transformation satisfying the following regularity condition, which enables us to subsequently obtain regularity conditions on the covariance function of the transformed Gaussian process.

\begin{condition} \label{cond:T}
Let $F: \mathbb{R} \to \mathbb{R}$ be a fixed non-constant  continuously differentiable function, with derivative $F'$.
We say that $F$ is sub-exponential and non-decreasing if:
\begin{itemize}
\item[i)] For all $t \in \mathbb{R}$, we have $| F(t) | \leq \Csup \exp{ ( \Csup |t|) }$ and $| F'(t) | \leq \Csup \exp{ ( \Csup |t|) }$;
\item[ii)] The function $F$ is non-decreasing on $\mathbb{R}$.
\end{itemize}
\end{condition}

In the following lemma, we show that the covariance function of a transformed Gaussian process satisfies Condition~\ref{cond:sub:exp:ass:pos}, when  Conditions \ref{cond:sub:exp:ass:pos} and \ref{cond:T} are satisfied, for the original process and for the transformation.

\begin{lemma} \label{lemma:sub:exp:ass:pos}
Assume that the stationary covariance function $k$ satisfies Condition~\ref{cond:sub:exp:ass:pos} and that the transformation $F$ satisfies Condition~\ref{cond:T}. 
Let $X$ be a zero-mean Gaussian process with covariance function $k$ and let $k'$ be the stationary covariance function of $F(X(\cdot))$.
Then, $k'$ satisfies Condition~\ref{cond:sub:exp:ass:pos}.
\end{lemma}

In the next lemma, we show that we can replace the condition of an increasing-transformation by the condition of a monomial transformation of even degree (with an additive constant).

\begin{lemma} \label{lemma:lambda:inf:R}
If a covariance function $k$ satisfies Condition~\ref{cond:sub:exp:ass:pos} (i) and if the Fourier transform  $\hat{k}$ of $k$ is strictly positive on $\mathbb{R}^d$, then $k$ satisfies Condition~\ref{cond:sub:exp:ass:pos} (ii). 
Furthermore, in this case, in Lemma~\ref{lemma:sub:exp:ass:pos}, Condition~\ref{cond:T} (ii) can be replaced by the condition $F(x) = x^{2r} + u$ for $r \in \mathbb{N}$, $u \in \mathbb{R}$ and $x \in \mathbb{R}$. 
\end{lemma}

\section{Two main technical results}
\label{section:two:main:technical:results}

\subsection{Transformed Gaussian process framework}
\label{subection:tranformed:GP:framework}

Throughout the paper, we will consider an unobserved latent Gaussian process $Z$ on $\mathbb{R}^d$ with $d \in \mathbb{N}$ fixed. Assume that $Z$ has zero-mean and stationary covariance function $k_{Z}$. We assume throughout that $k_Z$ satisfies Condition~\ref{cond:sub:exp:ass:pos}.

We consider a fixed transformation function $T$ satisfying Condition~\ref{cond:T}.
We assume that we observe the transformed Gaussian process $Y$, defined by $Y(s) = T(Z(s))$ for any $s \in \mathbb{R}^d$.

We assume throughout that the random field $Y$ has zero-mean. We remark that, for a non-linear transformation $F : \mathbb{R} \to \mathbb{R}$, the random variable $F(X(s))$ does not necessarily have zero mean for $s \in \mathbb{R}^d$. Hence, we implicitly assume that $T$ is of the form $F -  \mathbb{E}[F(Z(x))]$, where $F$ satisfies Condition~\ref{cond:T}. 
Note that $\mathbb{E}[F(Z(x))]$ is constant by stationarity and that, if $F$ satisfies Condition~\ref{cond:T} or the condition specified in Lemma~\ref{lemma:lambda:inf:R}, then $F -  \mathbb{E}[F(Z(x))]$ also satisfies these conditions. Here, as in many references, the assumption of a zero-mean for $Y$ is made by notational convenience and for the sake of brevity, and could be alleviated.

We let $k_{Y}$ be the covariance function of $Y$. We remark that, from Lemma~\ref{lemma:sub:exp:ass:pos}, $k_Y$ also satisfies Condition~\ref{cond:sub:exp:ass:pos}.

We let $(s_i)_{i \in \mathbb{N}}$ be the sequence of observation locations, with $s_i \in \mathbb{R}^d$ for $i \in \mathbb{N}$. We assume that $(s_i)_{i \in \mathbb{N}}$ satisfies Condition~\ref{cond:delta}. 

For $n \in \mathbb{N}$, we let $y= (y_1,\ldots,y_n)^\top = ( Y(s_1),\ldots,Y(s_n) )^\top$ be the (non-Gaussian) observation vector and $R = (k_Y(s_i-s_j))_{i,j=1,\ldots,n}$ be its covariance matrix. 

The problem of estimating the covariance function $k_Y$ from the observation vector $y$ is crucial and has been extensively studied in the Gaussian case (when $T$ is a linear function). Classically, we assume that $k_Y$ belongs to a parametric family of covariance functions. We will provide the asymptotic properties of two of the most popular estimators of the covariance parameters: the one based on the (Gaussian) maximum likelihood \cite{rasmussen06gaussian,stein99interpolation} and the one based on cross validation \cite{bachoclagnoux2017,bachoc2013cross,zhang10kriging}. To our knowledge, such properties are currently known only for Gaussian processes, and we will provide analogous properties in the transformed Gaussian framework.

\subsection{Bounds on the elements of inverse covariance matrices}

In the case of (non-transformed) Gaussian processes, one important argument for establishing the asymptotic properties of the maximum likelihood and cross validation estimators is to bound the largest eigenvalue of the inverse covariance matrix $R^{-1}$. Unfortunately, due to the non-linearity of the transformation $T$, such a bound on the largest eigenvalue is no longer sufficient in our setting. 

To circumvent this issue, we obtain in the following theorem stronger control over the matrix $R^{-1}$: we show that its coefficients decrease polynomially quickly with respect to the corresponding distance between observation locations. This theorem may have an interest in itself.

\begin{theorem} \label{theorem:decay:coeff:Rinverse}
Consider the setting of Section~\ref{subection:tranformed:GP:framework}.
For all fixed $0 < \tau < \infty$, we have, for all $n \in \mathbb{N}$ and $i,j = 1 , \ldots , n$
\[
\left| \left( R^{-1} \right)_{i,j} \right|
\leq
\frac{ \Csup }{ 1 + | s_i - s_j |^{d+\tau} },
\]
where $\Csup$ depends on $\tau$ but does not depend on $n,i,j$.
\end{theorem}

\subsection{Central limit theorem for quadratic forms of transformed Gaussian processes}

In the proofs on covariance parameter estimation of Gaussian processes, a central step is to show the asymptotic normality of quadratic forms of large Gaussian vectors. This asymptotic normality is established by diagonalizing the matrices of the quadratic forms. This diagonalization provides sums of squares of decorrelated Gaussian variables and thus sums of independent variables \cite{istas97quadratic,bachoc14asymptotic}.

In the transformed Gaussian case, one has to deal with quadratic forms involving transformations of Gaussian vectors. Hence, the previous arguments are not longer valid. To overcome this issue, we provide below a general central limit theorem for quadratic forms of transformed Gaussian vectors. This theorem may have an interest in itself.

This asymptotic normality result is established by considering a metric $d_w$  generating the topology of weak convergence on the set of Borel probability measures on Euclidean spaces (see, e.g., \cite{dudleyreal} p. 393). We prove that the distance between the sequence of the standardized distributions of the quadratic forms and Gaussian distributions decreases to zero when $n$ increases. The introduction of the metric $d_w$ enables us to formulate asymptotic normality results in cases when the sequence of  standardized variances of the quadratic forms does not necessarily converge as $n \to \infty$.

\begin{theorem} \label{theorem:generic:TCL}
Consider the setting of Section~\ref{subection:tranformed:GP:framework}.
Let $(A_n)_{n \in \mathbb{N}}$ be a sequence of matrices such that $A_n$ has dimension $n \times n$ for any $n \in \mathbb{N}$. Let $A = A_n$ for concision. Assume that for all $n \in \mathbb{N}$ and for all $i,j=1,\ldots,n$,
\[
| A_{i,j} |
\leq
\frac{ \Csup }{ 1 + | s_i - s_j |^{d + \Cinf }},
\]
where $\Csup$ does not depend on $i,j$.
Let 
\begin{align}
V_n = \frac{1}{n} y^\top A y. \label{eq:Vn}
\end{align}
Let $\mathcal{L}_n$ be the distribution of $ \sqrt{n} (V_n - \IE[V_n] )$.
Then, as $n \to \infty$, 
\[
d_w
\bigl(
\mathcal{L}_n
,
\cN
[
0,
n \var(V_n)
]
\bigr)
\to 0.
\]
In addition, the sequence $(n \var(V_n))_{n \in \mathbb{N}}$ is bounded.
\end{theorem}

\begin{remark}\label{remark1}
In the case where limits $E_{\infty}$ and $\sigma_{\infty}$ exist, such that
\[
\IE[V_n] - E_{\infty} = o(n^{-1/2})
\]
and
 the sequence  $(n \var(V_n))_{n \in \mathbb{N}}$ converges to a fixed variance $\sigma_{\infty}^2$, the result of Theorem~\ref{theorem:generic:TCL} can be written in the classical form
\[
\sqrt{n}
\left(
V_n - E_{\infty}
\right)
\overset{\mathcal{L}}{\to}
\cN[ 0 ,\sigma_{\infty}^2],
\]
as $n \to \infty$.
\end{remark}

\section{Estimation of a single variance parameter} \label{section:sigma}

We let $\sigma_0^2$ be the marginal variance of $Y$, that is $\Var(Y(s)) = \sigma_0^2$ for any $s \in \IR^d$. We let $ k_Y = \sigma_0^2 c_Y$ be the stationary covariance function of $Y$, where $c_Y$ is a correlation function. We assume that the same conditions as in Section~\ref{section:two:main:technical:results} hold. Then, the standard Gaussian maximum likelihood estimator of the variance parameter is
\[
\hat{\sigma}_{\text{ML}}^2
=
\frac{1}{n}
y^\top C^{-1} y, 
\]
where $C = (c_Y(s_i-s_j))_{1 \leq i,j \leq n}$.
One can simply show that $\mathbb{E}[  \hat{\sigma}_{\text{ML}}^2 ] = \sigma_0^2$ even though $y$ is not a Gaussian vector, since $y$ has mean vector $0$ and covariance matrix $\sigma_0^2 C$.
 Hence, a direct consequence of Theorems~\ref{theorem:decay:coeff:Rinverse} and~\ref{theorem:generic:TCL} is then that the maximum likelihood estimator is asymptotically Gaussian, with a $n^{1/2}$ rate of convergence, even though the transformed process $Y$ is not a Gaussian process.

\begin{corollary}\label{cor:hat:sigma}
Let $\mathcal{L}_n$ be the distribution of $ \sqrt{n} ( \hat{\sigma}_{\text{ML}}^2 -  \sigma_0^2 )$.
Then, as $n \to \infty$, 
\[
d_w
\bigl(
\mathcal{L}_n
,
\cN
[
0,
n \var( \hat{\sigma}_{\text{ML}}^2 -  \sigma_0^2 )
]
\bigr)
\to 0.
\]
In addition, the sequence $\bigl(n \var( \hat{\sigma}_{\text{ML}}^2 -  \sigma_0^2 )\bigr)_{n \in \mathbb{N}}$ is bounded.
\end{corollary}

The proof of Theorem~\ref{theorem:generic:TCL} actually allows us to study another estimator of the variance $\sigma_0^2$  of the form 
\[
\hat{\sigma}_{\text{ML},K}^2
=
\frac{1}{n}
y^\top C_K^{-1} y, 
\]
where $(C_K^{-1})_{i,j} = (C^{-1})_{i,j}  {1}_{ |s_i - s_j| \leq K }$.  The proof of Theorem~\ref{theorem:generic:TCL} then directly implies the following.

\begin{corollary} \label{cor:hat:sigma:K}
Let $K_n$ be any sequence of positive numbers tending to infinity.
Let $\mathcal{L}_{K_n,n}$ be the distribution of $ \sqrt{n} ( \hat{\sigma}_{\text{ML},K_n}^2 -  \sigma_0^2 )$.
Then, as $n \to \infty$, 
\[
d_w
\bigl(
\mathcal{L}_{K_n,n}
,
\cN
[
0,
n \var( \hat{\sigma}_{\text{ML},K_n}^2 -  \sigma_0^2 )
]
\bigr)
\to 0.
\]
In addition, we have
\[
n \var( \hat{\sigma}_{\text{ML},K_n}^2 -  \sigma_0^2 )
-
n \var( \hat{\sigma}_{\text{ML}}^2 -  \sigma_0^2 )
\to 0.
\]
\end{corollary}

The above corollary shows that one can taper the elements of $C^{-1}$ when estimating the variance parameter, and obtain the same asymptotic distribution of the error, as long as the taper range goes to infinity, with no rate assumption. This result may have an interest in itself, in view of the existing literature on covariance tapering for Gaussian processes under increasing-domain asymptotics \cite{furrer2016asymptotic,shaby12tapered}. We also remark that the computation costs of $\hat{\sigma}_{\text{ML},K}^2$ and $\hat{\sigma}_{\text{ML}}^2$ have the same orders of magnitude because $C^{-1}$ needs to be computed in both cases.

\section{General covariance}
\label{sec:estitheta}

\subsection{Framework}
\label{sub:sec:Framework:estitheta}

As in Section~\ref{subection:tranformed:GP:framework}, we consider a zero-mean Gaussian process $Z$  defined on $\mathbb{R}^d$ with covariance function $ k_Z$ satisfying Condition~\ref{cond:sub:exp:ass:pos}. Let  $Y$ be the random field defined for any $s \in \mathbb{R}^d$ by $Y(s) = T(Z(s))$, where $T$ is a fixed function  satisfying Condition~\ref{cond:T}. Furthermore we assume that $Y$ has zero-mean function and we recall that from Lemma~\ref{lemma:sub:exp:ass:pos}, its covariance function $k_Y$ also satisfies Condition~\ref{cond:sub:exp:ass:pos}. Finally, the sequence of observation locations $(s_i)_{i\in\IN}$ satisfies Condition~\ref{cond:delta}.

Let $\{ k_{Y,\theta} ; \theta \in \Theta \}$ be a parametric set of stationary covariance functions on $\IR^d$, with $\Theta$ a compact set of $\IR^p$. We consider the following condition on this parametric set of covariance functions.

\begin{condition} \label{cond:cov:Y:theta}
For all $s \in \mathbb{R}^d$, $k_{Y,\theta}(s)$ is three times continuously differentiable with respect to $\theta$, and we have
\begin{equation} \label{eq:ktheta:decay}
\sup_{\theta \in \Theta}
|k_{Y,\theta}(s)| 
\leq 
\Csup \exp{ (- \Cinf |s|) },
\end{equation}
\begin{equation} \label{eq:ktheta:smoothness}
\sup_{ \substack{
\theta \in \Theta
\\
\ell = 1,2,3
\\
i_1,\ldots,i_\ell = 1,\ldots,p
}}
\left|
  \frac{ \partial k_{Y,\theta}(s) }{ \partial \theta_{i_1} , \ldots , \partial \theta_{i_\ell} }
  \right| 
\leq 
\frac{\Csup}{ 1+ |s|^{d + \Cinf} }.
\end{equation}
\end{condition}
The smoothness condition in \eqref{eq:ktheta:smoothness} is classical and is assumed for instance in \cite{bachoc14asymptotic}. As discussed after Condition~\ref{cond:sub:exp:ass:pos}, milder versions of \eqref{eq:ktheta:decay} can be assumed for non-transformed Gaussian processes, but \eqref{eq:ktheta:decay} is satisfied by most classical families of covariance functions nonetheless.

The next condition, on the Fourier transforms of the covariance functions in the model, is standard.

\begin{condition} \label{cond:cov:Y:theta:fourier}
We let $\hat{k}_{Y,\theta}$ be the Fourier transform of $k_{Y,\theta}$. Then $\hat{k}_{Y,\theta}(s)$ is jointly continuous with respect to $\theta$ and $s$ and is strictly positive on $\Theta \times \IR^d$.
\end{condition}

Finally, the next condition means that we address the well-specified case  \cite{bachoc2013cross,bachoc2018asymptotic}, where the family of covariance functions does contain the true covariance function of $Y$. The well-specified case is considered in the majority of the literature on Gaussian processes.

\begin{condition} \label{cond:theta:0}
There exists $\theta_0$ in the interior of $\Theta$ such that $k_Y = k_{Y,\theta_0}$.
\end{condition}
In the next two subsections, we study  the asymptotic properties of two classical estimators (maximum likelihood and cross validation) for  the covariance  parameter $\theta_0$. The asymptotic properties of these estimators are already known for Gaussian processes and we extend them to the non-Gaussian process $Y$.
\subsection{Maximum Likelihood}

For $n \in \IN$, let $R_{\theta}$ be the $n \times n$ matrix $\bigl( k_{Y,\theta}(s_i-s_j) \bigr)_{i,j = 1,\ldots,n}$, and let
\begin{equation} \label{eq:hat:theta:ML}
\hat{\theta}_{\text{ML}}
\in 
\argmin_{\theta \in \Theta}
L_{\theta}
\end{equation}
with
 \[
  L_{\theta} =   
\frac{1}{n}
\left(
 \log( \det( R_{\theta} ) ) + y^\top R_{\theta}^{-1} y
 \right)
 \]
 be a maximum likelihood estimator. We will provide its consistency under the following condition.
 
 \begin{condition} \label{cond:asymptotic:identifiability:theta}
For all $\alpha >0$ we have
\[
\liminf_{n \to \infty}
\inf_{ || \theta - \theta_0 || \geq \alpha }
\frac{1}{n}
\sum_{i,j=1}^n
\bigl(
k_{Y,\theta}( s_i-s_j )
-
k_{Y,\theta_0}( s_i-s_j )
\bigr)^2
> 0.
\]
\end{condition}

Condition~\ref{cond:asymptotic:identifiability:theta} can be interpreted as a global indentifiability condition.
It implies in particular that two different covariance parameters yield two different distributions for the observation vector.
It is used in several studies, for instance \cite{bachoc14asymptotic}.

\begin{theorem} \label{theorem:consistency:thetaML}
Consider the setting of Section~\ref{sub:sec:Framework:estitheta} for $Z$, $T$, $Y$ and $(s_i)_{i \in \mathbb{N}}$.
Assume that Conditions 
\ref{cond:cov:Y:theta}, \ref{cond:cov:Y:theta:fourier}, \ref{cond:theta:0} and \ref{cond:asymptotic:identifiability:theta} hold. Then, as $n \to \infty$,
\[
\hat{\theta}_{\text{ML}}
\overset{p}{\to}
\theta_0.
\]
\end{theorem}

\begin{remark} \label{remark:explaination:ML}
The (Gaussian) maximum likelihood estimator $\hat{\theta}_\text{ML}$ in Theorem~\ref{theorem:consistency:thetaML} is not, strictly speaking, a maximum likelihood estimator, because the distribution of the observation vector $y$ is non-Gaussian. If the transformation $T$ is known, and assuming that the covariance function of the Gaussian process $Z$ belongs to a parametric set $\{ k_{Z,\alpha} , \alpha \in \mathcal{A}\}$, a (non-Gaussian) maximum likelihood estimator $\hat{\alpha}_{Z,\text{ML}}(y)$ could be defined. When $T$ is bijective, $y$ is a fixed invertible transformation of $(Z(s_1),\ldots,Z(s_n))$ and so this maximum likelihood estimator $\hat{\alpha}_{Z,\text{ML}}(y)$ coincides with the standard Gaussian maximum likelihood estimator based on $Z$. Asymptotically, it is known that $\hat{\alpha}_{Z,\text{ML}}(y)$ converges to the true parameter $\alpha_0$ for which $Z$ has covariance function $k_{Z,\alpha_0}$. In Theorem~\ref{theorem:consistency:thetaML}, we show that  $\hat{\theta}_{\text{ML}}$ similarly converges to the true parameter $\theta_0$ for which $Y$ has covariance function $k_{Y,\theta_0}$. Notice that $\alpha_0$ and $\theta_0$ usually do not coincide. For instance, if the covariance models $\{ k_{Y,\theta} ; \theta \in \Theta \}$ and $\{ k_{Z,\alpha} ; \alpha \in \mathcal{A} \}$ are both $\{\sigma^2 e^{ - \rho ||\cdot||  }\}$, then if $\alpha_0 = (\sigma_0^2 , \rho_0)$, by Mehler's formula \cite{Azais2009level},
\[
\Cov( Y(s + \tau) , Y(s) )
=
2 \Cov( Z(s + \tau) , Z(s) )^2
= 2 \sigma_0^4 e^{- 2 \rho_0 || \tau ||}
\]
and thus $\theta_0 = (2 \sigma_0^4 , 2 \rho_0)$. Hence, the exact (non-Gaussian) maximum likelihood estimator $\hat{\alpha}_{Z,\text{ML}}(y)$ based on the knowledge of $T$ and modeling the covariance function of $Z$ and the pseudo (Gaussian) maximum likelihood estimator $\hat{\theta}_{\text{ML}}$ that ignores $T$ and models the covariance function of $Y$ estimate covariance parameters of different natures and do not have the same limit.
\end{remark}

\begin{condition} \label{cond:asymptotic:identifiability:local:theta}
For any $(\alpha_1,\ldots,\alpha_p) \neq (0,\ldots,0)$, we have
\[
\liminf_{n \to \infty}
\frac{1}{n}
\sum_{i,j=1}^n
\Bigg(
\sum_{\ell = 1}^p
\alpha_{\ell}
\frac{ \partial k_{Y,\theta_0}( s_i-s_j )}{ \partial \theta_\ell }
\Bigg)\bigg.^2
> 0.
\]
\end{condition}

Condition~\ref{cond:asymptotic:identifiability:local:theta} can be interpreted as a regularity condition and as a local indentifiability condition around $\theta_0$.
In the next theorem, we provide the asymptotic normality of the maximum likelihood estimator. In this theorem, the matrices $M_{\theta_0}$ and $\Sigma_{\theta_0}$ depend on the number of observation locations.

\begin{theorem} \label{theorem:CLT:thetaML}
Consider the setting of Section~\ref{sub:sec:Framework:estitheta} for $Z$, $T$, $Y$ and $(s_i)_{i \in \mathbb{N}}$.
Assume that Conditions~\ref{cond:cov:Y:theta}, \ref{cond:cov:Y:theta:fourier}, \ref{cond:theta:0}, \ref{cond:asymptotic:identifiability:theta} and \ref{cond:asymptotic:identifiability:local:theta} hold.
Let $M_{\theta_0}$ be the $p \times p$ matrix defined by 
\[
(M_{\theta_0})_{i,j} = \frac{1}{n}  \trace \Bigl(  R_{\theta_0}^{-1} \frac{\partial R_{\theta_0}}{\partial \theta_i}  R_{\theta_0}^{-1} \frac{\partial R_{\theta_0}}{\partial \theta_j }\Bigr).
\]
 Let $\Sigma_{\theta_0}$ be the $p \times p$ covariance matrix defined by $(\Sigma_{\theta_0})_{i,j} = \Cov (  n^{1/2} \partial L_{\theta_0} / \partial \theta_i , n^{1/2} \partial L_{\theta_0} / \partial \theta_j )$.
Let $\mathcal{L}_{\theta_0,n}$ be the distribution of $\sqrt{n} ( \hat{\theta}_{\text{ML}} - \theta_0 )$.
Then, as $n \to \infty$,
\begin{equation} \label{eq:TCL:theta}
d_w
\left(
\mathcal{L}_{\theta_0,n}
,
\cN
[
0
,
M_{\theta_0}^{-1}
\Sigma_{\theta_0}
M_{\theta_0}^{-1}
]
\right)
\to 0.
\end{equation}
In addition, 
\begin{equation} \label{eq:in:TCL:theta:lambda1}
\limsup_{n \to \infty}
\lambda_1(  M_{\theta_0}^{-1}
\Sigma_{\theta_0}
M_{\theta_0}^{-1} )
<
+
\infty.
\end{equation}
\end{theorem}

\begin{remark}
If the sequences of matrices $M_{\theta_0}$ and $\Sigma_{\theta_0}$ converge as $n\to\infty$, then $\sqrt n(\hat{\theta}_{\text{ML}}-\theta_0)$ converges in distribution to a fixed centered Gaussian distribution where the limiting covariance matrix is given by 
$$
\lim_{n\to\infty}M_{\theta_0}^{-1}
\Sigma_{\theta_0}
M_{\theta_0}^{-1}.
$$ 
\end{remark}

Conditions~\ref{cond:asymptotic:identifiability:theta} and \ref{cond:asymptotic:identifiability:local:theta} involve the model of covariance functions $\{ k_{Y,\theta} ; \theta \in \Theta \}$ and the sequence of observation locations $(s_i)_{i \in \IN}$ but not the transformation $T$.
They are further discussed, in a different context, in \cite{bachoc2018gaussian}.
 We believe that these conditions are mild. For instance, Conditions~\ref{cond:asymptotic:identifiability:theta} and \ref{cond:asymptotic:identifiability:local:theta} hold when the sequence of observation locations $(s_i)_{i \in \IN}$ is a randomly perturbed regular grid, as in \cite{bachoc14asymptotic}.

\begin{lemma}[see \cite{bachoc14asymptotic}] \label{lemma:example:ML}
For $i \in \mathbb{N}$, let $s_i = g_i + \epsilon_i$, where $(g_i)_{i \in \IN}$ is a sequence with, for $N \in \IN$, $\{g_1,\ldots,g_{N^d}\} = \{ (i_1,\ldots,i_d) ; i_1 = 1,\ldots,N, \ldots , i_d = 1,\ldots , N \}$ and where $(\epsilon_i)_{i \in \IN}$ is a sequence of $i.i.d.$ random variables with distribution on $[-1/2+\delta,1/2-\delta]^d$ with $0<\delta<1/2$. 
Then, Condition~\ref{cond:asymptotic:identifiability:theta} holds almost surely, provided that, for $\theta \neq \theta_0$, there exists $i \in \IZ^d$ for which $k_{Y,\theta}( i + \epsilon_1 - \epsilon_2 )$ and $k_{Y,\theta_0}( i + \epsilon _1 - \epsilon_2 )$ are not almost surely equal.

Furthermore, Condition~\ref{cond:asymptotic:identifiability:local:theta} holds almost surely, provided that for $(\alpha_1,\ldots,\alpha_p) \neq (0 , \ldots ,0)$, there exists $i \in \IZ^d$ for which $ \sum_{\ell=1}^p  \alpha_\ell \partial k_{Y,\theta_0}( i + \epsilon_1 - \epsilon_2 ) / \partial \theta_\ell$ is not almost surely equal to zero.
\end{lemma}

\subsection{Cross  Validation} \label{subsection:CV}
We consider the cross validation estimator consisting of minimizing the sum of leave-one-out square errors.

Since the leave-one-out errors do not depend on the variance $k_{Y,\theta}(0)$, we introduce some additional notation. In Subsections~\ref{subsection:CV} and~\ref{sub:sec:joint}, we let $\Theta = [\sigma_{\inf}^2 , \sigma_{\sup}^2] \times \cS$ where $0 < \sigma_{\inf}^2 < \sigma_{\sup}^2 < \infty$ are fixed and where $\mathcal{S}$ is compact in $\mathbb{R}^{p-1}$. We let
 $\theta = (\sigma^2 , \psi)$ with $\sigma_{\inf}^2 \leq \sigma^2 \leq \sigma_{\sup}^2$ and $\psi \in \mathcal{S} $. We assume that for $\theta \in \Theta$, $k_{Y,\theta} = \sigma^2 c_{Y , \psi}$, with $c_{Y,\psi}$ a stationary correlation function. Similarly, we let $\theta_0 = (\sigma_0^2 , \psi_0)$. For $\psi \in \mathcal{S}$, we let $C_{\psi} = \bigl( c_{Y, \psi} (s_i - s_j) \bigr)_{i,j=1,\ldots,n}$.
Cross validation is defined for  $n \in \IN$ by 

\begin{equation} \label{eq:hat:theta:CV}
\hat{\psi}_{\text{CV}}
\in 
\argmin_{\psi \in \mathcal{S}}
CV_{\psi},
\end{equation}
with
\[
CV_{\psi} = \frac{1}{n} y^\top C_\psi^{-1} \diag(C_\psi^{-1})^{-2} C_\psi^{-1}y,
\]
where $\diag(M)$ is obtained by setting the off-diagonal elements of a square matrix $M$ to zero.
The criterion $CV_\psi$ is the sum of leave-one-out square errors, as is shown for instance in \cite{dubrule83cross,bachoc2013cross}.

 The asymptotic behaviour of $\hat{\psi}_{\text{CV}}$ was studied in the Gaussian framework in \cite{bachoc14asymptotic} and under increasing-domain asymptotics. We also remark that, in the Gaussian framework, a modified leave-one-out criterion was studied  in \cite{bachoclagnoux2017} in the case of infill asymptotics.

The next identifiability condition is also made in \cite{bachoc14asymptotic}.

\begin{condition} \label{cond:asymptotic:identifiability:theta:CV}
For all $\alpha >0$, we have
\[
\liminf_{n \to \infty}
\inf_{ || \psi - \psi_0 || \geq \alpha }
\frac{1}{n}
\sum_{i,j=1}^n
\left(
c_{Y,\psi}(s_i - s_j)
-
c_{Y,\psi_0} (s_i - s_j)
\right)^2
> 0.
\]
\end{condition}

The  next theorem provides the consistency of the cross validation estimator.

\begin{theorem} \label{theorem:consistency:thetaCV}
Consider the setting of Section~\ref{sub:sec:Framework:estitheta} for $Z$, $T$, $Y$ and $(s_i)_{i \in \mathbb{N}}$.
Assume that Conditions 
\ref{cond:cov:Y:theta}, \ref{cond:cov:Y:theta:fourier}, \ref{cond:theta:0} and \ref{cond:asymptotic:identifiability:theta:CV} hold. Then, as $n \to \infty$,
\[
\hat{\psi}_{\text{CV}}
\overset{p}{\to}
\psi_0.
\]
\end{theorem}

The next condition is a local identifiability condition.

\begin{condition} \label{cond:asymptotic:identifiability:local:theta:CV}
For any $(\alpha_1,\ldots,\alpha_{p-1}) \neq (0,\ldots,0)$, we have
\[
\liminf_{n \to \infty}
\frac{1}{n}
\sum_{i,j=1}^n
\Biggl(
\sum_{\ell = 1}^{p-1}
\alpha_{\ell}
\frac{ \partial}{ \partial \psi_\ell } \left(
c_{Y,\psi_0}( s_i-s_j )
\right)\Bigg)\bigg.^2
> 0.
\]
\end{condition}

In the next theorem, we provide the asymptotic normality of the cross validation estimator. In this theorem, the matrices $N_{\psi_0}$ and $\Gamma_{\psi_0}$ depend on the number of observation locations. 

\begin{theorem} \label{theorem:CLT:thetaCV}
Consider the setting of Section~\ref{sub:sec:Framework:estitheta} for $Z$, $T$, $Y$ and $(s_i)_{i \in \mathbb{N}}$.
Assume that Conditions~\ref{cond:cov:Y:theta}, \ref{cond:cov:Y:theta:fourier}, \ref{cond:theta:0}, \ref{cond:asymptotic:identifiability:theta:CV} and \ref{cond:asymptotic:identifiability:local:theta:CV} hold.
Let $N_{\psi_0}$ be the $(p-1) \times (p-1)$ matrix defined by 
\begin{align*}
    (N_{\psi_0})_{i,j} =& - \frac{8}{n} \trace \Bigl( \frac{\partial C_{\psi_0}}{\partial \psi_j} C_{\psi_0}^{-1} \diag(C_{\psi_0}^{-1})^{-3} \diag \Bigl( C_{\psi_0}^{-1} \frac{\partial C_{\psi_0}}{\partial \psi_i} C_{\psi_0}^{-1} \Bigr) C_{\psi_0}^{-1} \Bigr)\\
    & + \frac{2}{n} \trace \Bigl( \frac{\partial C_{\psi_0}}{\partial \psi_j} C_{\psi_0}^{-1} \diag(C_{\psi_0}^{-1})^{-2} C_{\psi_0}^{-1} \frac{\partial C_{\psi_0}}{\partial \psi_i} C_{\psi_0}^{-1} \Bigr)\\
    & + \frac{6}{n} \trace \Bigl( \diag(C_{\psi_0}^{-1})^{-4} \diag\Bigl( C_{\psi_0}^{-1} \frac{\partial C_{\psi_0}}{\partial \psi_i} C_{\psi_0}^{-1} \Bigr) \diag\Bigl( C_{\psi_0}^{-1} \frac{\partial C_{\psi_0}}{\partial \psi_j} C_{\psi_0}^{-1} \Bigr) C_{\psi_0}^{-1} \Bigr).
\end{align*}
Let $\Gamma_{\psi_0}$ be the $(p-1) \times (p-1)$ covariance matrix defined by $(\Gamma_{\psi_0})_{i,j} = \Cov ( n^{1/2} \partial CV_{\psi_0} / \partial \psi_i ,\linebreak[1] n^{1/2} \partial CV_{\psi_0} / \partial \psi_j )$.
Let $\mathcal{Q}_{\psi_0,n}$ be the distribution of $\sqrt{n} ( \hat{\psi}_{\text{CV}} - \psi_0 )$.
Then, as $n \to \infty$,
\begin{equation} \label{eq:TCL:theta:CV}
d_w
\left(
\mathcal{Q}_{\psi_0,n}
,
\cN
[
0
,
N_{\psi_0}^{-1}
\Gamma_{\psi_0}
N_{\psi_0}^{-1}
]
\right)
\to 0.
\end{equation}
In addition,
\begin{equation} \label{eq:in:TCL:theta:lambda1:CV}
\limsup_{n \to \infty}
\lambda_1(  N_{\psi_0}^{-1}
\Gamma_{\psi_0}
N_{\psi_0}^{-1} )
<
+
\infty.
\end{equation}
\end{theorem}

Similarly as for maximum likelihood, Condition~\ref{cond:asymptotic:identifiability:theta:CV} is a global identifiability condition for the correlation function. In the same way, Condition~\ref{cond:asymptotic:identifiability:local:theta:CV} is a local identifiability condition for the correlation function around $\psi_0$.
The sequence of observation locations presented in Lemma~\ref{lemma:example:ML} also satisfies Conditions~\ref{cond:asymptotic:identifiability:theta:CV} and \ref{cond:asymptotic:identifiability:local:theta:CV} by replacing $k_{Y,\theta}$ with $c_{Y,\psi}$.

We remark that the conditions for cross validation imply those for maximum likelihood.

\begin{lemma} \label{lem:cond:CV:implies:ML}
Condition~\ref{cond:asymptotic:identifiability:theta:CV} implies Condition~\ref{cond:asymptotic:identifiability:theta} and Condition~\ref{cond:asymptotic:identifiability:local:theta:CV} implies Condition~\ref{cond:asymptotic:identifiability:local:theta}.
\end{lemma}

\subsection{Joint asymptotic normality}
\label{sub:sec:joint}

From Theorems~\ref{theorem:CLT:thetaML} and \ref{theorem:CLT:thetaCV}, both the maximum likelihood and cross validation estimators converge at the standard parametric rate $n^{1/2}$.
Let us write $\hat{\theta}_{\text{ML}} = (\hat{\sigma}_{\text{ML}}^2 , \hat{\psi}_{\text{ML}})$.
In the case where $T$ is the identity function (that is, where we observe Gaussian processes instead of transformed Gaussian processes), numerical experiments tend to show that $\hat{\psi}_\text{ML}$ is more accurate than $\hat{\psi}_\text{CV}$ \cite{bachoc14asymptotic}. Indeed, when $T$ is the identity function, maximum likelihood is based on the Gaussian probability density function of the observation vector.

In contrast, when $T$ is not the identity function, $\hat{\psi}_\text{ML}$ is an M-estimator based on a criterion which does not coincide with the observation probability density function anymore. Hence, it is conceivable that $\hat{\psi}_\text{CV}$ could become more accurate than $\hat{\psi}_\text{ML}$. Furthermore, it is possible that using linear combinations of these two estimators could result in a third one with improved accuracy \cite{lavancier2016general,bates1969combination}.

Motivated by this discussion, we now provide a joint central limit theorem for the maximum likelihood and cross validation estimators.

\begin{theorem} \label{theorem:CLT:joint}
Consider the setting of Section~\ref{sub:sec:Framework:estitheta} for $Z$, $T$, $Y$ and $(s_i)_{i \in \mathbb{N}}$.
Assume that Conditions~\ref{cond:cov:Y:theta}, \ref{cond:cov:Y:theta:fourier}, \ref{cond:theta:0}, \ref{cond:asymptotic:identifiability:theta:CV} and \ref{cond:asymptotic:identifiability:local:theta:CV} hold.
Let $D_{\theta_0}$ be the $(2p-1) \times (2p-1)$ block diagonal matrix with first $p \times p$ block equal to $M_{\theta_0}$ and second $(p-1) \times (p-1)$ block equal to $N_{\psi_0}$, with the notation of Theorems \ref{theorem:CLT:thetaML} and \ref{theorem:CLT:thetaCV}.
 Also let 
$\Psi_{\theta_0}$ be the $(2p-1) \times (2p-1)$ covariance matrix of the vector $n^{1/2}( \partial L_{\theta_0} / \partial \theta , \partial CV_{\psi_0} / \partial \psi$).

Let $\mathcal{Q}_{\theta_0,n}$ be the distribution of 
\[
\sqrt{n} 
\begin{pmatrix}
\hat{\theta}_{\text{ML}} - \theta_0 \\
\hat{\psi}_{\text{CV}} - \psi_0
\end{pmatrix}.
\]
Then, as $n \to \infty$,
\begin{equation} \label{eq:TCL:joint}
d_w
\left(
\mathcal{Q}_{\theta_0,n}
,
\cN
[
0
,
D_{\theta_0}^{-1}
\Psi_{\theta_0}
D_{\theta_0}^{-1}
]
\right)
\to 0.
\end{equation}
In addition, 
\begin{equation} \label{eq:in:TCL:theta:lambda1:joint}
\limsup_{n \to \infty}
\lambda_1(  D_{\theta_0}^{-1}
\Psi_{\theta_0}
D_{\theta_0}^{-1} )
<
+
\infty.
\end{equation}
\end{theorem}

\begin{remark} \label{remark:linear:combination}
From Theorem~\ref{theorem:CLT:joint}, considering any $C^1$ function $f$ from $\mathcal{S}^2 \to \mathcal{S}$ such that $f(\psi,\psi) = \psi$ for any $\psi \in \mathcal{S}$, and applying the classical delta method, we obtain the asymptotic normality of the new estimator $f( \hat{\psi}_{\text{ML}} , \hat{\psi}_{\text{CV}} )$. A classical choice for $f$ is $f(\psi_1,\psi_2) = \lambda \psi_1 + (1 - \lambda) \psi_2$, which leads to linear aggregation \cite{lavancier2016general,bates1969combination}. We remark that selecting an optimal $\lambda$ leading to the smallest asymptotic covariance matrix necessitates an estimation of the asymptotic covariance matrix in Theorem~\ref{theorem:CLT:joint}. We leave this as an open problem for further research.
\end{remark}

\section{Illustration} \label{section:simulation:result}

In this section we numerically illustrate the convergence of different estimators as stated in Corollary~\ref{cor:hat:sigma} and Theorems~\ref{theorem:CLT:thetaML}, \ref{theorem:CLT:thetaCV} and \ref{theorem:CLT:joint}. We use the following simulation setup in dimension $d=2$. For $n=4(m+1/2)^2$ we define the grid $\{-m,\ldots,m\}^2$ and add i.i.d. variables with uniform distribution on $[-0.4,0.4]^2$ to obtain $n$ observation points. Thus, we have a distance of at least $\Delta=0.2$ between the individual observation locations. The zero-mean Gaussian process $Z$ has stationary covariance function $k_Z(s)=\sigma_0^2\exp(-||s||/\rho_0)$, $s\in\IR^2$ and we will denote this reference case as the Gaussian case throughout. We define the zero-mean process $Y=T(Z)=Z^2-\sigma_0^2$ and we will denote this case as the non-Gaussian case. Recall that $k_Y(s)=2k_Z(2s)=2\sigma_0^4\exp(-2||s||/\rho_0)$ (see Remark~\ref{remark:explaination:ML}). We set the marginal variance to $\sigma_0^2=1.5$ and the range to $\rho_0=2$. Hence, in the non-Gaussian case, the marginal variance of $Y$ is $2\sigma_0^4=4.5$ with a range of $\rho_0/2$ which is equal to a half of that of~$Z$.

To start, we consider the maximum likelihood estimates of the marginal variance parameters, when the range of $Y$ or $Z$ is assumed to be known, i.e.,  Corollary~\ref{cor:hat:sigma}. 
Figure~\ref{fig:1} illustrates the empirical densities of $\mathcal{L}_n$ in Corollary \ref{cor:hat:sigma} for $n=100, 400$ and $900$ observation locations based on $N=2500$ replicates. For moderate $n$ sizes and in the non-Gaussian case, the asymptotic variance $\sigma^2_\infty:= n \Var( \hat{\sigma}_{\text{ML}}^2 - 2 \sigma_0^2 )$ (Corollary \ref{cor:hat:sigma}) can be calculated based on~\eqref{eq:nVarVn} and using 
\def\Sigmaf{k}
\begin{align*}
 \Cov( y_iy_j,y_ky_l )&= 4(\Sigmaf_{i,k}^2\Sigmaf_{j,l}^2+\Sigmaf_{i,l}^2\Sigmaf_{j,k}^2)\\&\quad+\def\Sigmaf{k}  
      16(  \Sigmaf_{i,k}\Sigmaf_{i,l}\Sigmaf_{j,k}\Sigmaf_{j,l} +  
                 \Sigmaf_{i,j}\Sigmaf_{i,l}\Sigmaf_{j,k}\Sigmaf_{k,l} +  
                 \Sigmaf_{i,j}\Sigmaf_{i,k}\Sigmaf_{j,l}\Sigmaf_{k,l}  ),
\end{align*}
with $\Sigmaf_{i,j}=k_Z(s_i-s_j)$.
The above display follows from tedious computations based on Isserlis' theorem. 
The densities in red in Figure~\ref{fig:1} are based on the asymptotic distribution $\cN[0,\sigma^2_\infty]$. In the non-Gaussian case and for $n\geq 400$, the calculation of  $\sigma^2_\infty$ is computationally prohibitive, so $\sigma_{\infty}^2$ has instead been approximated by the empirical variance with the corresponding densities indicated in green. As expected, the convergence for the Gaussian case is faster than for the non Gaussian case. But in both situations, the results behave nicely.

\begin{figure}[ht]
  \centering
  \includegraphics[width=\textwidth]{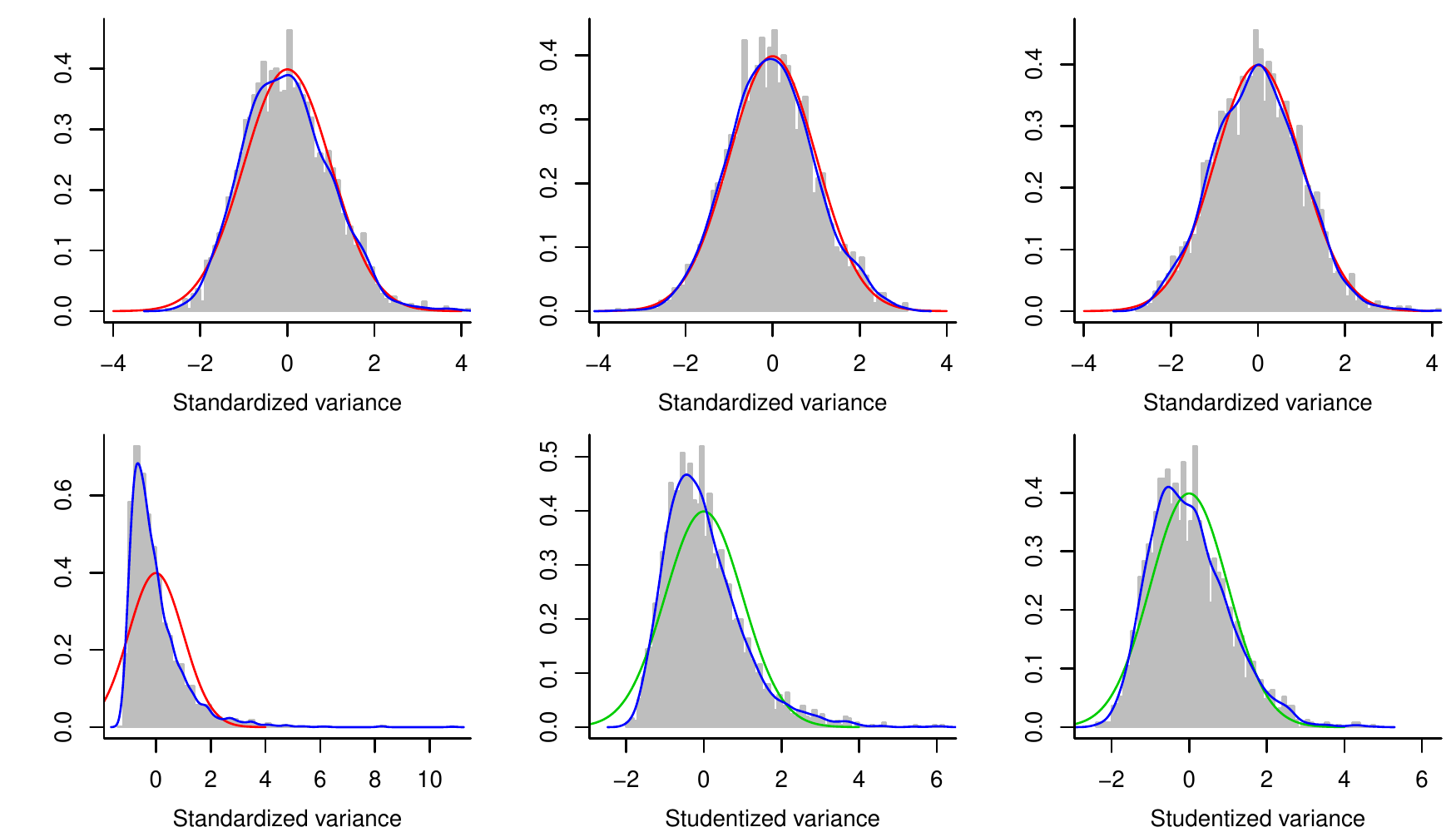}  
  \caption{Histograms of standardized and studentized variance maximum likelihood estimates. Blue: empirical density (kernel density estimates), red: asymptotic density, green: asymptotic density based on the empirical variance. The top row shows results for the process $Z$ (the Gaussian case), the bottom row for the transformed process $Y=Z^2-\sigma_0^2$ (the non-Gaussian case). The columns are for $n=100,400,900$. All panels are based on $N=2500$ replicates.}
  \label{fig:1}
\end{figure}

We now turn to Theorem~\ref{theorem:CLT:thetaML} and consider the bivariate variance and range maximum likelihood estimation.
That is, we consider the two-dimensional maximum likelihood estimates of $(\sigma_0^2,\rho_0)$ in the Gaussian case and of $(2 \sigma_0^2,\rho_0/2)$ in the non-Gaussian case.
 Again, we do not observe many surprises. Skewness of the empirical distribution is slightly higher compared to the single variance parameter estimation, and convergence is slightly slower, as is illustrated in Figure~\ref{fig:2}. 

For the general setting, when estimating jointly the variance and range parameter, the asymptotic bivariate covariance matrix is challenging to compute (see Theorem \ref{theorem:CLT:thetaML}) and thus 
Figure~\ref{fig:2} illustrates the empirical densities and densities based on the empirical bivariate covariance matrix.

\begin{figure}[t]
  \centering
  \includegraphics[width=\textwidth]{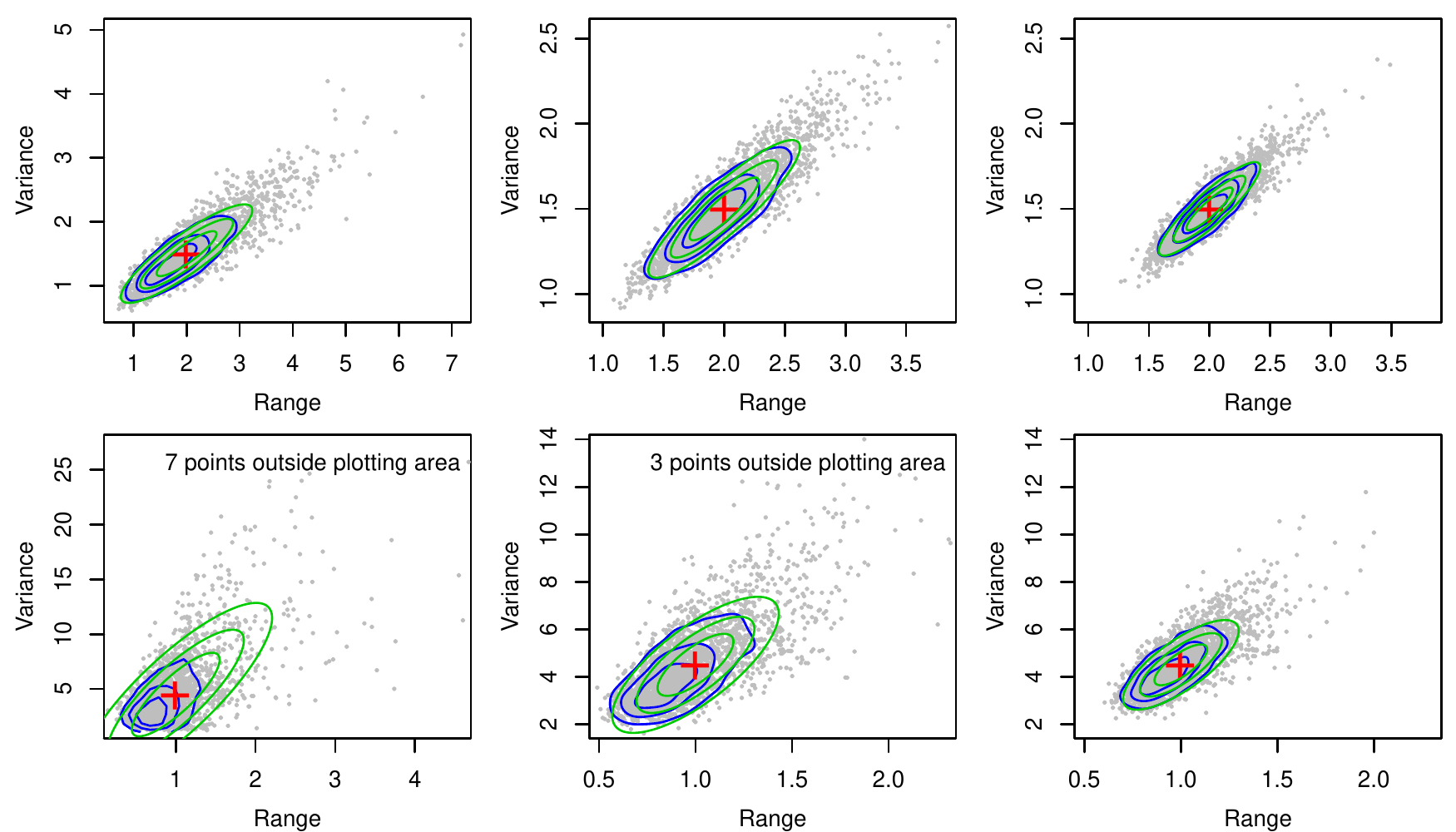}  
  \caption{Scatter plots of variance  and range maximum likelihood estimates. Blue: contour lines of kernel density estimates; green: contour lines of asymptotic density based on the empirical bivariate covariance matrix; red cross: true mean (true parameter values). The top row shows results for the process $Z$ (Gaussian case), the bottom row for the transformed process $Y=Z^2-\sigma_0^2$ (non-Gaussian case). The columns are for $n=100,400,900$. All panels are based on $N=2500$ replicates.}
  \label{fig:2}
\end{figure}

\medskip

We now consider not only maximum likelihood estimation of the variance and range, but also cross validation estimation of the range (see the beginning of Section \ref{subsection:CV}).
We have observed that the range estimates based on cross validation are much more variable, and in many situations the maximum was attained at the (imposed) boundary. Here we used the bound $[2/15, 12]$, i.e., smaller than the minimal distance between two observation locations and 6 (resp. 12) times the diameter of the observation points of $Z$ (resp. $Y$). 
Estimates at or close to the boundary indicate convergence issues and would imply a second, possibly manual, inspection. For the reported results, we eliminated all cross validation cases that yielded estimates outside $[0.14,11.4]$.

Figure~\ref{fig:MSE} shows the mean squared error,  squared bias and variance of $\hat{\sigma}_{\text{ML}}^2$, $\hat{\rho}_{\text{ML}}$ and $\hat{\rho}_{\text{CV}}$ under different settings. For maximum likelihood, we consider the univariate case (one parameter is estimated while the other is known) and the bivariate case (both parameters are jointly estimated). For cross validation, only the range parameter is estimated (see the beginning of Section \ref{subsection:CV}), and thus only the univariate case is considered. The mean squared error is dominated by the variance component. Univariate maximum likelihood estimation for Gaussian cases have low bias and the lowest variance (top left and right panel). Joint maximum likelihood estimation has a somewhat larger variance than individual estimation.
Surprisingly, cross validation for Gaussian cases has a higher variance compared to cross validation for non-Gaussian cases.

\begin{figure}[ht]
  \centering
  \includegraphics[width=\textwidth]{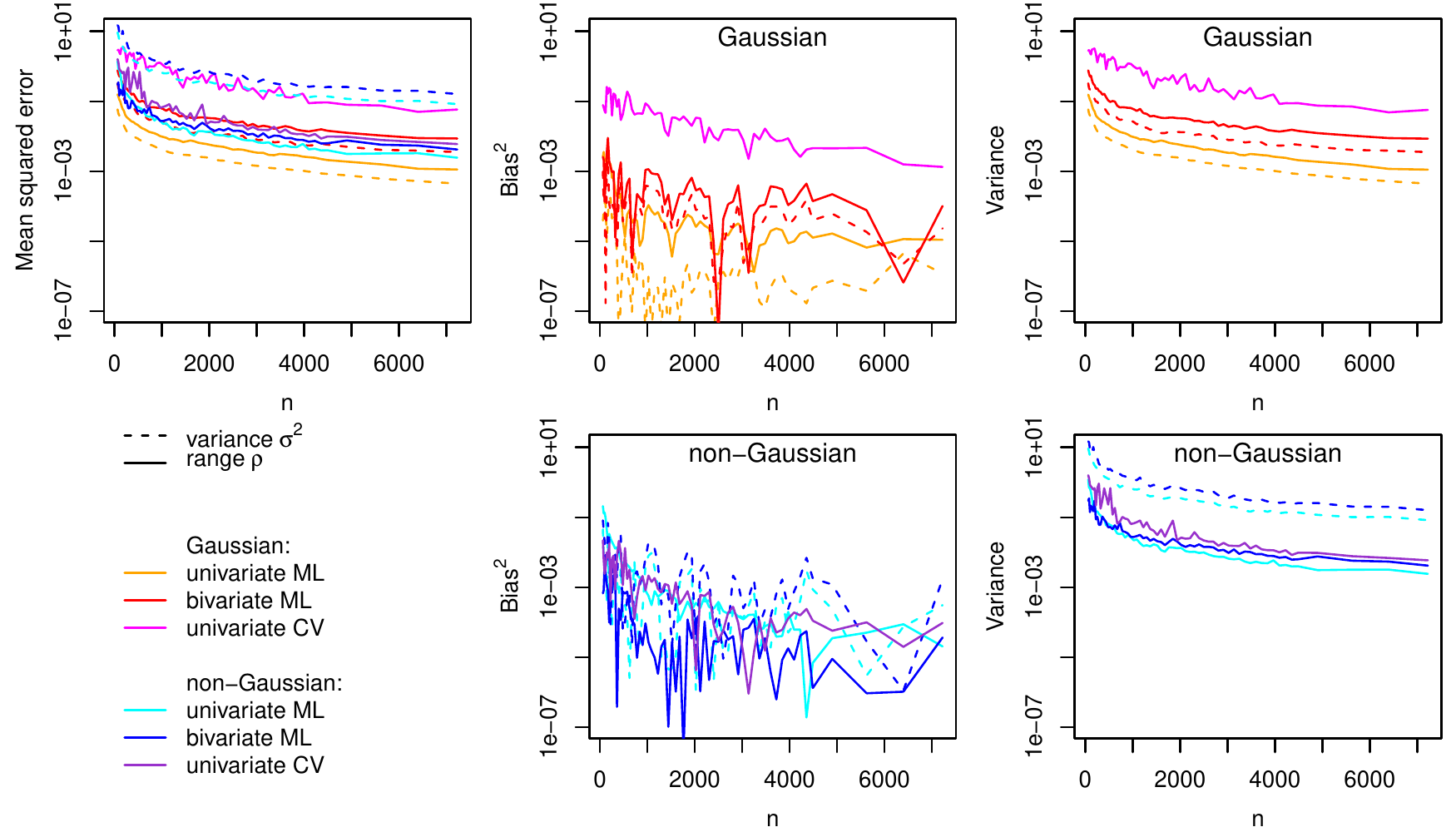}  
  \caption{Mean squared error (top left), squared bias (middle column) and variance (right column) as a function of $n$ for different settings in
    log-scale.  The variance parameter is represented in dashed lines, the range parameter in solid lines. Process $Z$ is shown by reddish colors (Gaussian case),
    transformed process  $Y=Z^2-\sigma_0^2$ in blueish colors (non-Gaussian case).  The panels are based on $N=250$ replicates that did not reveal any convergence issues.}
  \label{fig:MSE}
\end{figure}

Recall that Theorems \ref{theorem:CLT:thetaML}, \ref{theorem:CLT:thetaCV} and \ref{theorem:CLT:joint} show that, as $n$ increases, the distribution of the standardized estimation error is close to a Gaussian distribution in terms of the metric $d_w$. 
In Figure~\ref{fig:3}, we illustrate this by computing one-dimensional Wasserstein distances between the empirical distribution of the standardized estimation errors and Gaussian distributions. 
The figure shows the Wasserstein distance ($p=1$) as a function of the number of observation locations $n$ for individual parameters and for specific bivariate settings (similarly as for Figure \ref{fig:MSE}). In each case, the samples have been centered around the true mean (true parameter values) and standardized by an empirical standard deviation ($n$-weighted average over all the samples). Their empirical distribution is compared
to the standardized Gaussian distribution.
The top left panel shows that the densities of the cross validation parameters are converging slowest whereas their mean squared error is comparable (see Figure \ref{fig:MSE}); the densities are highly skewed and thus lead to much larger Wasserstein distances compared to the distributions of the maximum likelihood derived parameters.
For the bivariate maximum likelihood estimation the marginal distributions have very similar  Wasserstein distances; in the center panels: the dashed and solid colored lines are visually hardly separable.  As suggested by the individual panels of Figures~\ref{fig:1} and~\ref{fig:2}, convergence in the Gaussian case is much faster compared to the non-Gaussian case.
The right column of Figure~\ref{fig:3} illustrates the joint asymptotic normality of the range parameter estimators by maximum likelihood and cross validation.
The gray lines there illustrate Theorem \ref{theorem:CLT:joint} and are Wasserstein distances for linear combinations of the range estimates by maximum likelihood and cross validation, i.e., $\lambda\hat{\rho}_\text{ML}+(1-\lambda)\hat{\rho}_\text{CV}$ for $\lambda=j/10$, $j=1,\dots,9$.
The highly skewed distribution of the cross validation-estimated range parameter for Gaussian processes is clearly visible. In the non Gaussian case, the effect of the skewness is less pronounced since the maximum likelihood is skewed as well.

\begin{figure}[ht]
  \centering
  \includegraphics[width=\textwidth]{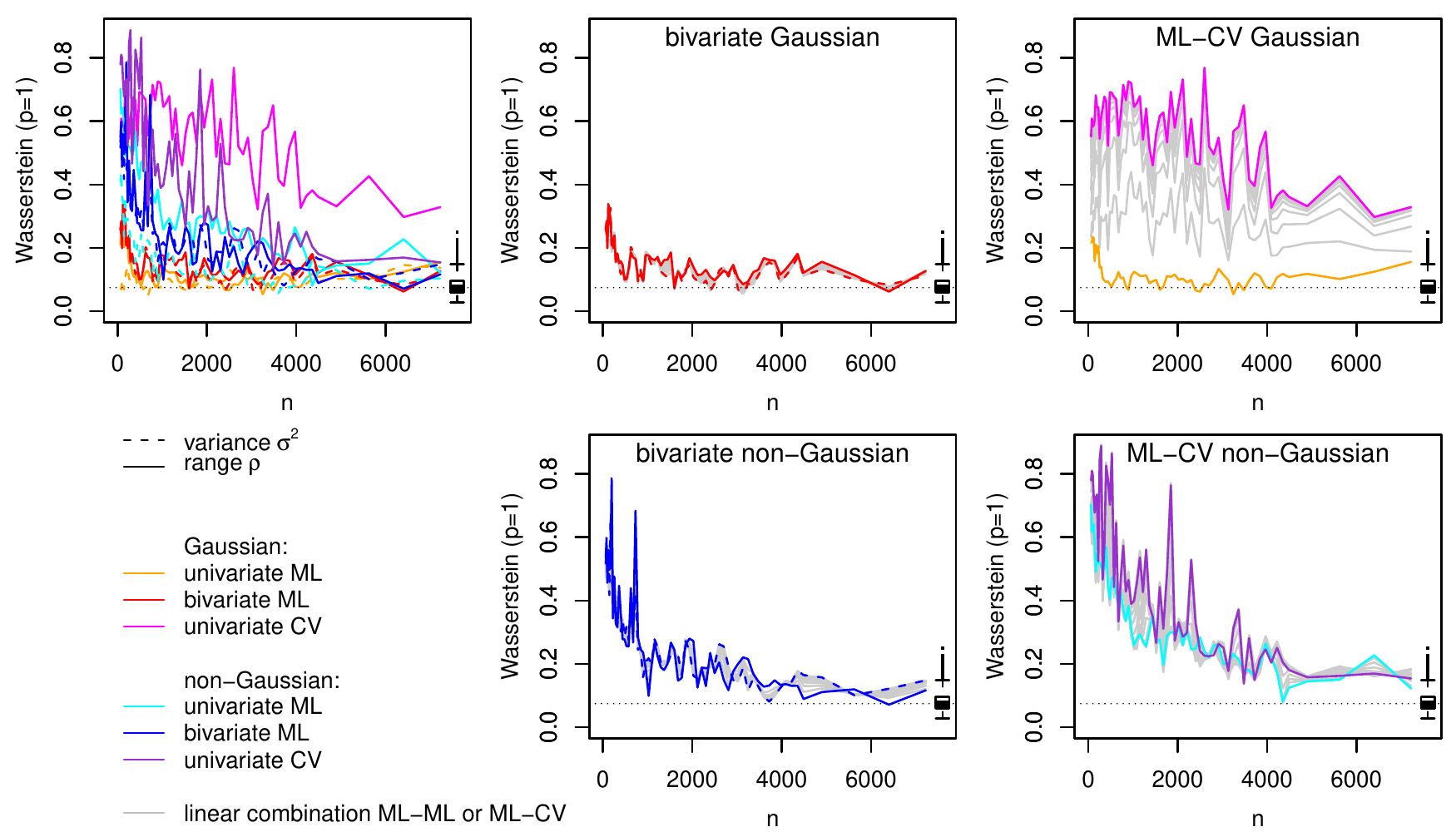}  
  \caption{Wasserstein distance ($p=1$). Top left: marginal for each parameter. Center panels: bivariate estimation of range and variance by maximum likelihood as in Theorem~\ref{theorem:CLT:thetaML} and linear combinations thereof; right panels: univariate estimation of the range parameter by maximum likelihood and cross validation and with linear combinations of these two estimators as in Theorem~\ref{theorem:CLT:joint}. Top middle and right panels: Gaussian cases. Lower row panels: non-Gaussian case. The colors and line styles follow those in Figure~\ref{fig:MSE}. The gray lines are Wasserstein distances for estimates based on linear combinations of maximum likelihood (center column) and of maximum likelihood and cross validation (right column). The panels are based on $N=250$ replicates that did not reveal any convergence issues. The small boxplot on the right in each panel shows the Wasserstein distance for sample size $n=250$ of $10000$ realizations of $\cN[0,1]$; the horizontal dotted line shows the median thereof.}
  \label{fig:3}
\end{figure}

\section{Conclusion}

We have shown that the covariance parameters of transformed Gaussian processes can be estimated by cross validation and Gaussian maximum likelihood, with the same rate of convergence as in the case of non-transformed Gaussian processes. In particular, Gaussian maximum likelihood works well asymptotically, despite the fact that the observations do not have a Gaussian distribution. Hence Gaussian maximum likelihood is here robust with respect to non-Gaussian data. This provides the first step of a theoretical validation of the use of Gaussian maximum likelihood in frequent cases where the data are non-Gaussian.

In future research, it would be interesting to extend the results of this paper to other classes of non-Gaussian random fields rather than only transformed Gaussian processes. In addition, the asymptotic analysis of estimators of the transformation of transformed Gaussian processes is of great interest.

\appendix

\section{Proofs}\label{app}

\subsection{Technical results}

\begin{lemma} \label{lemma:finite:mean:function:gaussian:vector}
Let $q \in \IN$ be fixed. Let $g: \IR^q \to \IR^+$ be fixed and satisfy $g(x) \leq \Csup \exp( \Csup |x| )$. Let $W$ be a Gaussian vector of dimension $q$. Then $\IE[ g(W) ] < \infty$.
\end{lemma}

\begin{proof}
Without loss of generality, we can assume that $W_i$ has variance $1$ for $i=1,\ldots,q$. We let $w_i$ be the mean of $W_i$ for $i=1,\ldots,q$. We have, for $t > 1$,
\begin{align*}
\IP( g(W) \geq t )
& \leq \IP( \Csup \exp(\Csup \max_{i=1,\ldots,q} |W_i| ) \geq t )
\\
& \leq \sum_{i=1}^q
\IP( \Csup \exp(\Csup |W_i| ) \geq t )
\\
& =
\sum_{i=1}^q
\IP( |W_i| \geq (1/\Csup) \log(t / \Csup) )
\\
& \leq
2 \sum_{i=1}^q
\IP( W \geq (1/\Csup) \log(t / \Csup) - |w_i| ),
\end{align*}
where $W \sim \cN[0,1]$.
From the Gaussian tail inequality, we obtain, for 
\[
t \geq \Csup \exp( \Csup ( \max_{i=1,\ldots,q} |w_i| +1 ) ),
\]
that
\[
\IP( g(W) \geq t )
\leq
\frac{2 q }{ \sqrt{2 \pi} }
\exp \Bigl( - (1/2) \bigl( (1/\Csup) \log(t / \Csup) - \max_{i=1,\ldots,q} |w_i| \bigr)^2   \Bigr).
\]
The function of $t$ above is clearly summable as $t \to + \infty$. Hence, we have $\IE[ g(W) ] = \int_{0}^{\infty} \IP( g(W) \geq t ) < + \infty$.
\end{proof}

\begin{lemma} \label{lemma:cov:muliple:products}
Let $X$ be centered Gaussian process with covariance function
$k_X$ satisfying Condition~\ref{cond:sub:exp:ass:pos}. Let $F$ satisfy Condition~\ref{cond:T}. 
Let $W$ be the spatial process $F(X(\cdot))$ and assume that $W$ is centered.
Let $(x_i)_{i \in \mathbb{N}}$ satisfy Condition~\ref{cond:delta}.

Then, we have, for any $r_1,r_2 \in \IN$ and $\Delta \geq 0$,
\[
\sup_{ \substack{ i_1,\ldots,i_{r_1} \in \IN \\
j_1,\ldots,j_{r_2} \in \IN \\
\min_{a=1,\ldots,r_1, b=1,\ldots,r_2} |x_{i_a} - x_{j_b}| \geq \Delta \hspace*{-5mm}
 } }
 \left|
 \Cov( W(x_{i_{1}}) \dots W(x_{i_{r_1}}) , W(x_{j_{1}}) \dots W(x_{j_{r_2}}) )
 \right|
 \leq 
 \Csup e^{- \Cinf \Delta},
\]
where $\Csup$ and $\Cinf$  depend on $r_1 , r_2$ but not on $\Delta$.
\end{lemma}

\begin{proof}
Let $\Delta \geq 0$, $ i_1,\ldots,i_{r_1} \in \IN$ and
$j_1,\ldots,j_{r_2} \in \IN$ such that
$\min_{a=1,\ldots,r_1, b=1,\ldots,r_2} |x_{i_a} - x_{j_b}| \geq \Delta$.

Let $\epsilon_3 \sim \cN[ 0 , I_{r_1} ]$ and $ \epsilon_4 \sim \cN[ 0 , I_{r_2} ]$, $\epsilon_3$ and $ \epsilon_4$ being independent. Let $R$ be the $r_2 \times r_1$ matrix $( k_{X}(  x_{j_a} , x_{i_b} ) )_{a,b}$, let $C_1$ be the $r_1 \times r_1$ matrix $( k_{X}(  x_{i_a} , x_{i_b} ) )_{a,b}$
and let  $C_2$ be the $r_2 \times r_2$ matrix $( k_{X}(  x_{j_a} , x_{j_b} ) )_{a,b}$. Let $M = R C_1^{-1}$. 
Let $K$ be a matrix square root of $C_2 - R C_1^{-1} R^\top$. Let $K_1$ be the unique symmetric matrix square root of $C_1$.

Then the vector $((K_1 \epsilon_3)^\top, (M K_1 \epsilon_3 + K \epsilon_4)^\top )$ has the same distribution as 
\[
\bigl(X(x_{i_{1}}), \dots, X(x_{i_{r_1}}) , X(x_{j_{1}}), \dots, X(x_{j_{r_2}})\bigr).
\]
For $i=1,2$, for $x = (x_1,\ldots,x_{r_i}) \in \IR^{r_i}$, let $f_i(x) = F(x_1) \cdots F(x_{r_i}) \in \IR$. 
Then we have
\begin{align*}
\Cov( W(x_{i_{1}}) \cdots W(x_{i_{r_1}}) , W(x_{j_{1}}) \cdots W(x_{j_{r_2}}) )
& =
\cov \left( f_1( K_1 \epsilon_3 ) , f_2( M K_1 \epsilon_3 + K \epsilon_4  )  \right)
.
\end{align*}

By a Taylor expansion, there exists a random vector $\epsilon_5$ belonging to the segment with endpoints $K \epsilon_4$ and $M K_1 \epsilon_3 +  K \epsilon_4$ such that, with $G_2(\epsilon_5)$ the gradient column vector of $f_2$ at $\epsilon_5$, we have
\[
f_2( M K_1 \epsilon_3 + K \epsilon_4  ) 
=
f_2( K \epsilon_4 )
+
( M K_1 \epsilon_3 )^\top
G_2(\epsilon_5).
\]
This yields
\begin{align*}
\left|
\Cov( W(x_{i_{1}}) \ldots W(x_{i_{r_1}}) , W(x_{j_{1}}) \ldots W(x_{j_{r_2}}) )
\right|
& =
\left|
\Cov( f_1( K_1 \epsilon_3 ) , \epsilon_3^\top K_1^\top M^\top G_2( \epsilon_5 )  ) 
\right|
\\
& \leq 
\sqrt{ \IE [ f_1^2 ( K_1 \epsilon_3 ) ] }
\sqrt{ \IE [ (   \epsilon_3^\top K_1^\top M^\top G_2( \epsilon_5 ) )^2 ]  }.
\end{align*}
From Condition~\ref{cond:sub:exp:ass:pos} ii) and from the equivalence of norms, we obtain $||  M ||_{op} \leq  \Csup e^{- \Cinf \Delta}$ and $||K_1||_{op} \leq \Csup$, where $\Csup$ and $\Cinf$ do not depend on $ i_1,\ldots,i_{r_1},j_1,\ldots,j_{r_2},\Delta$.
By equivalence of norms, we then obtain, with $\Csup$ and $\Cinf$ not depending on  $ i_1,\ldots,i_{r_1}$, $j_1,\ldots,j_{r_2},\Delta$,
\[
\left|
\Cov( W(x_{i_{1}}) \ldots W(x_{i_{r_1}}) , W(x_{j_{1}}) \ldots W(x_{j_{r_2}}) )
\right|
\leq
\Csup
e^{ - \Cinf \Delta}
\sqrt{ \IE [ f_1^2 ( K_1 \epsilon_3 ) ] 
 \IE [ (   || \epsilon_3 || \ ||G_2( \epsilon_5 )|| )^2 ]  }.
\]
Now, 
\[
||\epsilon_5  || \leq ||M K_1 \epsilon_3|| +  ||K \epsilon_4|| 
\leq \Csup \left(
 || \epsilon_3|| +  || \epsilon_4|| 
 \right),
\]
from Condition~\ref{cond:sub:exp:ass:pos} ii),
where, again, $\Csup$ does not depend on  $ i_1,\ldots,i_{r_1},j_1,\ldots,j_{r_2},\Delta$.
Furthermore, $||K_1 \epsilon_3|| \leq \Csup || \epsilon_3 ||$. 
Eventually, we have
\begin{align*}
\big|
\Cov( W(x_{i_{1}})& \cdots W(x_{i_{r_1}}) , W(x_{j_{1}}) \cdots W(x_{j_{r_2}}) )
\big|
\\ & \leq
\Csup
\exp( - \Cinf \Delta)
\sqrt{ \IE [ f_1^2 (K_1  \epsilon_3 ) ] }
\sqrt{ \IE \Bigl[ \bigl(   || \epsilon_3 || \sup_{ ||x|| \leq  \Csup(
 || \epsilon_3|| +  || \epsilon_4|| 
 )  } 
 ||G_2 ( x )||
     \bigr)^2  \Bigr]  }.
\end{align*}
From Condition~\ref{cond:T} i), we have $|f_1(K_1 x)| \leq \Csup e^{\Csup ||x||}$ and $ ||G_2 ( x )|| \leq \Csup e^{\Csup ||x||}$ where $\Csup$ does not depend on $x$ and  $ i_1,\ldots,i_{r_1},j_1,\ldots,j_{r_2},\Delta$. Hence 
the above square roots are finite from Lemma~\ref{lemma:finite:mean:function:gaussian:vector} and do not depend on 
$ i_1,\ldots,i_{r_1}$,
$j_1,\ldots,j_{r_2}$ and $\Delta$. This concludes the proof.
\end{proof}

\begin{lemma} \label{lemma:covijkl}
Consider the setting of Section~\ref{subection:tranformed:GP:framework}, that is $k_Z$ satisfies Condition~\ref{cond:sub:exp:ass:pos} and
$T$ satisfies Condition~\ref{cond:T}.
For $n \in \mathbb{N}$ and $i,j=1,\ldots,n$ we have
\[
\sum_{k,l = 1, \ldots, n}
\left| \Cov \left( y_i  y_j , y_k  y_l \right) \right|
\leq \Csup,
\]
where $\Csup$ does not depend on $n,i,j$.
\end{lemma}

\begin{proof}

We let $d( a , (b,c) ) = \min(|a-b|,|a-c|)$ for $a,b,c \in \IR^d$.
It is enough to show that 
\begin{equation} \label{in:proof:TCL:enough:to:show}
\left| \Cov( y_i  y_j , y_k  y_l  ) \right|
\leq \Csup \exp \Big( - \Cinf \max \big( d(s_k,(s_i,s_j)) , d(s_l,(s_i,s_j)) \big) \Big).
\end{equation}
Indeed, let, for $t \geq 0$, $N_{i,j,t}$ be the number of pairs $(k,l)$, with $1 \leq k,l \leq n$ such that 
\[
t \leq \max( d(s_k,(s_i,s_j)) , d(s_l,(s_i,s_j)) ) \leq t+1.
\]
From Condition~\ref{cond:delta}, we can show that we have $\sup_{n \in \IN , i,j=1,\ldots,n} N_{i,j,t} \leq \Csup t^{2d}$.  Hence, we have
\begin{align*}
\sum_{i,j = 1}^n\!\!
\exp \Bigl( - \Cinf \max \bigl( d(s_k,(s_i,s_j)) , d(s_l,(s_i,s_j)) \bigr) \Bigr)
& \leq 
\sum_{k=0}^{+ \infty}
\Csup (k+1)^{2d}
\exp \left( - \Cinf k  \right)
< + \infty.
\end{align*}
Thus, \eqref{in:proof:TCL:enough:to:show} implies the result of the lemma and it suffices to prove \eqref{in:proof:TCL:enough:to:show}.

Let $i,j,k,l \in \{1,\ldots,n\}$ and let $\Delta = \max( d(s_k,(s_i,s_j)) , d(s_l,(s_i,s_j)) )$. By symmetry, we can consider that $d(s_k,(s_i,s_j)) = \Delta$. 

If $|s_k - s_l| \leq | s_i - s_l |$ and
$|s_k - s_l| \leq | s_j - s_l |$, then $| s_i - s_l|  \geq \Delta /2$ and $| s_j - s_l | \geq \Delta /2$. Hence, we can apply Lemma~\ref{lemma:cov:muliple:products} with distance $\Delta/2$ to obtain $| \Cov( y_k y_l , y_i y_j)| \leq \Csup \exp( - \Cinf \Delta/2  )$, where $\Csup$ and $\Cinf$ do not depend on $n,i,j,k,l,\Delta$.

If $|s_i - s_l| \leq | s_k - s_l |$ and
$|s_i - s_l| \leq | s_j - s_l |$, then $| s_k - s_l|  \geq \Delta /2$. We then have
\begin{align} \label{eq:for:proof:summability:covijkl:covijkl}
\Cov( y_i y_j , y_k y_l )
& = 
\IE[y_i y_j y_k y_l]
- \IE[y_i y_j] \IE[y_k y_l]
\nonumber
\\
& =
\Cov( y_iy_jy_l , y_k )
- \Cov( y_i , y_j ) \Cov( y_k , y_l )
\end{align}
since it is assumed that $Y$ has zero-mean. In \eqref{eq:for:proof:summability:covijkl:covijkl}, the first and third covariances are bounded in absolute value by  $\Csup \exp( - \Cinf \Delta/2  )$ from Lemma~\ref{lemma:cov:muliple:products}, because $| s_k - s_l|  \geq \Delta /2$, $| s_k - s_i|  \geq \Delta /2$ and $| s_k - s_j | \geq \Delta /2$. Hence we have
$| \Cov( y_k y_l , y_i y_j)| \leq \Csup \exp( - \Cinf \Delta/2  )$, where $\Csup$ and $\Cinf$ do not depend on $n,i,j,k,l,\Delta$.

If $|s_j - s_l| \leq | s_k - s_l |$ and
$|s_j - s_l| \leq | s_i - s_l |$, we obtain the same bound by symmetry. We have thus considered all possible cases and the proof of \eqref{in:proof:TCL:enough:to:show} is concluded.
\end{proof}

In the context of Theorem~\ref{theorem:generic:TCL}, the following lemma provides an approximation of $V_n$, based on replacing $A$ by a sparse matrix. We remark that a similar approximation was shown in a time series context in \cite{Neumann13}. Nevertheless, we find that our assumptions on the random field $Y$ are more transparent and interpretable than the assumptions in \cite{Neumann13}, where cumulants are used. Because of these differences of assumptions, our proof of the following lemma differs from that in \cite{Neumann13}.

\begin{lemma} \label{lemma:A:sparse:approximation}
Let, for $K,n \in \IN$, $A^{(K)}$ be the $n \times n$ matrix defined by
\[
A^{(K)}_{i,j} = A_{i,j} 1_{ |s_i - s_j| \leq K }.
\]
Then, under the same assumptions as in Lemma~\ref{lemma:covijkl}, we have
\[
\sup_{n \in \mathbb{N}}
n
\Var
\Bigl(
\frac{1}{n}
y^\top
A
y 
-
\frac{1}{n}
y^\top
A^{(K)}
y 
\Bigr)
\to_{K \to \infty} 0.
\]
\end{lemma}

\begin{proof}
For any $K,n \in \mathbb{N}$ we have
\begin{align*}
n
\Var
\Bigl(
\frac{1}{n}
y^\top
A
y 
-
\frac{1}{n}
y^\top
A^{(K)}
y 
\Bigr)
& = 
\frac{1}{n}
\sum_{i,j,k,l=1}^n
(A - A^{(K)})_{i,j}
(A - A^{(K)})_{k,l}
\Cov( y_i y_j, y_k y_k )
         .
\end{align*}
We observe that $|(A - A^{(K)})_{k,l}|$ is equal to $0$ or is smaller than $\Csup / ( 1+  K^{d+\Cinf} )$ by assumption. Hence we have
\begin{align*}
n
\Var
\Bigl(
\frac{1}{n}
y^\top
A
y 
-
\frac{1}{n}
y^\top
A^{(K)}
y 
\Bigr)
& \leq 
\Csup \frac{1}{1+K^{d + \Cinf}}
\frac{1}{n}
\sum_{i,j,k,l=1}^n
| (A - A^{(K)})_{i,j} |
| \Cov( y_i y_j, y_k y_k )|
\\
& \leq 
\Csup \frac{1}{1+K^{d + \Cinf}}
\frac{1}{n}
\sum_{i,j=1}^n
| (A - A^{(K)})_{i,j} |
\\
& 
\leq
\Csup
\frac{1}{1+K^{d + \Cinf}}
\max_{i=1,\ldots,n}
\sum_{j=1}^n
\frac{ 1 }{ 1 + | s_i - s_j |^{d + \Cinf }}
\\
&
\leq \Csup
\frac{1}{1+K^{d + \Cinf}},
\end{align*}
where we have used Lemma~\ref{lemma:covijkl} and where we have observed that $|(A - A^{(K)})_{i,j}|$ is equal to $0$ or is smaller than $\Csup / \bigl( 1+  |s_i - s_j|^{d+\Cinf} \bigr)$. 
We have also used Lemma 4 in \cite{furrer2016asymptotic} for the last inequality above.
All the above constants $\Csup$ and $\Cinf$ naturally do not depend on $n$, so the lemma is proved.
\end{proof}

\begin{lemma} \label{lemma:alpha:mixing}
Consider the setting of Section~\ref{subection:tranformed:GP:framework}, that is $k_Z$ satisfies Condition~\ref{cond:sub:exp:ass:pos} and
$T$ satisfies Condition~\ref{cond:T}.
Let $a,b \in \mathbb{N}$.
For $i \in \{ 1,\ldots,a \}$, let $\alpha_i \in \mathbb{N}$ and let $I(i,1),\ldots,I(i,\alpha_i) \in \{1 , \ldots,n\}$. For $j \in \{ 1,\ldots,b \}$, let $\beta_j \in \mathbb{N}$ and let $J(j,1),\ldots,J(j,\beta_j) \in \{1 , \ldots,n\}$.
For $i=1,\ldots,a$, let $f_i$ be a function from $\mathbb{R}^{\alpha_i}$ to $\mathbb{R}$.
For $j=1,\ldots,b$, let $g_j$ be a function from $\mathbb{R}^{\beta_j}$ to $\mathbb{R}$.
For $i=1,\ldots,a$, let $v^{(i)} = (f_i(Z(s_{I(i,1)}),\ldots,Z(s_{I(i,\alpha_i)}))$.
For $j=1,\ldots,b$, let $w^{(j)} = (g_j(Z(s_{J(j,1)}),\ldots,Z(s_{J(j,\beta_j)}))$.
 Let 
\begin{align} \label{eq:def:alpha:mixing}
\begin{split}
  &\!\!\alpha \bigl(   \{ v^{(1)} , \ldots,v^{(a)} \} ,
 \{ w^{(1)} , \ldots,w^{(b)} \}  \bigr)
\\  
 & =
 \sup \Bigl\{
 \bigl|
 \IP(A \cap B)
 - \IP(A) \IP(B)
 \bigr|
 ;
 A \in \sigma( \{v^{(1)} , \ldots,v^{(a)} \} )
 ,
 B \in \sigma( \{ w^{(1)} , \ldots,w^{(b)} \} )
 \Bigr\},\!
\end{split}
\end{align}
where, for any set of random variables $\{\epsilon_1,\ldots,\epsilon_r\}$, $\sigma( \{\epsilon_1,\ldots,\epsilon_r\} )$ is the sigma algebra generated by the random variables $\{\epsilon_1,\ldots,\epsilon_r\}$.
Let 
\[
\Delta = \inf_{ \substack{
i \in \{1,\ldots,a\} \\
j \in \{1,\ldots,b\} \\
\tilde{i} \in \{1 , \ldots , \alpha_i\} \\
\tilde{j} \in \{1 , \ldots , \beta_j\}  
} }
|  s_{I(i,\tilde{i})} -  s_{J(j,\tilde{j})} |.
\]
Then, we have
\[
\alpha(   \{ v^{(1)} , \ldots,v^{(a)} \} ,
 \{ w^{(1)} , \ldots,w^{(b)} \}) 
 \leq 
 \Csup e^{ - \Cinf \Delta },
\]
where $\Csup$ and $\Cinf$ may depend on $a, \alpha_1,\ldots,\alpha_a$ but do not depend on $b$, $( J(j,\tilde{j}) )_{ j=1,\ldots,b,\tilde{j}=1,\ldots,\beta_j }$ and $\Delta$.
\end{lemma}

\begin{proof}
Let $\mathcal{I} =  \{ I(i,\tilde{i}) ; i=1,\ldots,a,\tilde{i}=1,\ldots,\alpha_i \}$ and let
$\mathcal{J} =  \{ J(j,\tilde{j}) ; j=1,\ldots,b,\tilde{j}=1,\ldots,\beta_j \}$. 
In \eqref{eq:def:alpha:mixing}, any of the events $A$ (resp. $B$) is an event defined on the set of random variables $\{ Z( s_{i} ) \}_{i \in \mathcal{I}}$ (resp. $\{ Z( s_{j} ) \}_{j \in \mathcal{J}}$). We thus obtain
\begin{align*}
\alpha \bigl(  \{ v^{(1)} , &\ldots,v^{(a)} \} ,
 \{ w^{(1)} , \ldots,w^{(b)} \}  \bigr)
 \\
 & \leq
 \alpha \bigl(
 \{ Z( s_{i} ) \}_{i \in \mathcal{I}}
 ,
 \{ Z( s_{j} ) \}_{j \in \mathcal{J}}
 \bigr)
 \\
 & :=
 \sup \Bigl\{
 \bigl|
 \IP(A \cap B)
 - \IP(A) \IP(B)
 \bigr|
 ;
 A \in \sigma( \{ Z( s_{i} ) \}_{i \in \mathcal{I}} )
 ,
 B \in \sigma(\{ Z( s_{j} ) \}_{j \in \mathcal{J}} )
 \Bigr\}.
\end{align*}

Let $ I_1<\dots <I_{\bar{\alpha}} $ be such that $\{ I_1,\ldots,I_{\bar{\alpha}} \} = \mathcal{I}$ and let $v_Z = (Z(s_{I_1}) , \ldots , Z(s_{I_{\bar{\alpha}}}))^\top$.
Let $ J_1<\dots <J_{\bar{\beta}} $ be such that $\{ J_1,\ldots,J_{\bar{\beta}} \} = \mathcal{J}$ and let $w_Z = (Z(s_{J_1}) , \ldots , Z(s_{J_{\bar{\beta}}}))^\top$.
From Lemma 1 in Section 2.1 of \cite{doukhan94mixing}, we have
\begin{align} \label{eq:in:lemma:alpha:v:w}
  \begin{split}
     \alpha \bigl(
 \{ Z( s_{i} ) \}_{i \in \mathcal{I}}
 ,&
 \{ Z( s_{j} ) \}_{j \in \mathcal{J}}
 \bigr)
 \\
& 
 \leq
 \sup \Bigl\{
 \bigl|
 \Cov( v^\top v_{Z} , w^\top w_Z  )
\bigr|
;
\Var(  v^\top v_{Z} ) = 1
,
\Var(  w^\top w_{Z} ) = 1
 \Bigr\}.
\end{split}\end{align}
Let $v$ and $w$ be vectors belonging to the set in \eqref{eq:in:lemma:alpha:v:w}. The smallest eigenvalues of the covariance matrices of $v_Z$ and $w_Z$ are larger that a constant $\Cinf$, not depending on $\mathcal{I}$ and $\mathcal{J}$, since $k_Z$ satisfies Condition~\ref{cond:sub:exp:ass:pos}.
Thus we have
\[
1 = \Var(  v^\top v_{Z} )
 = v^\top \Cov(v_Z) v
 \geq 
 \Cinf ||v||^2.
\]
It follows that $||v||^2 \leq \Csup $, where  $\Csup$ does not depend on $\mathcal{I}$, $\mathcal{J}$ and $\Delta$. Similarly $||w||^2 \leq \Csup$.

We have
\begin{align*}
\Cov^2( v^\top v_Z , w^\top w_Z ) 
& \leq
||\Cov( v^\top v_Z ,  w_Z )||^2
||w||^2
\\
&
\leq 
\Csup
\sum_{ j=1 , \ldots ,  \bar{\beta} }
\Cov( v^\top v_Z ,  Z( s_{J_j} ) )^2
\\
& 
\leq
\Csup
||v||^2
\sum_{i=1,\ldots,\bar{\alpha}}
\sum_{ j=1 , \ldots ,  \bar{\beta} }
\Cov(  Z(s_{I_i}) , Z( s_{J_j} ) )^2
\\
& 
\leq 
\Csup
\sum_{i=1,\ldots,\bar{\alpha}}
\sum_{ \substack{j=1,\ldots,n \\ |s_j - s_{i}| \geq \Delta }}
e^{ - \Cinf | s_{I_i} - s_j |},
\end{align*}
by definition of $\Delta$, since $k_Z$ satisfies Condition~\ref{cond:sub:exp:ass:pos} and where $\Csup$ and $\Cinf$ do not depend on $b$, $\mathcal{J}$ and $\Delta$. For any $i \in \{1 , \ldots ,n\}$, the number of indices $j \in \{1 , \ldots ,n\}$ such that $\tilde{\Delta} \leq |s_i- s_j| \leq \tilde{\Delta} + 1$ is smaller than $\Csup \tilde{\Delta}^d$, from Condition~\ref{cond:delta} and where $\Csup$ only depends on $d$. This yields
\begin{align*}
\Cov^2( v^\top v_Z , w^\top w_Z ) 
&
\leq
\Csup 
(\alpha_1+\ldots+\alpha_a)
\sum_{k=0}^{+ \infty}
\Csup ( \Delta+k )^d
e^{ - \Cinf | \Delta + k |}
\\
&
\leq
\Csup
e^{ - \Cinf | \Delta | /2}
(\alpha_1+\ldots+\alpha_a)
\sum_{k=1}^{+ \infty}
( \Delta+k )^d
e^{ - \Cinf | \Delta + k | / 2}
\\
&
\leq 
\Csup
e^{ - \Cinf | \Delta |}
\end{align*}
where the different $\Csup$ and $\Cinf$ 
do not depend on $b$, $\mathcal{J}$ and $\Delta$. This concludes the proof from \eqref{eq:in:lemma:alpha:v:w}. 
\end{proof}

\begin{lemma} \label{lem:product:matrix}
Consider a sequence $(x_i)_{i \in \IN}$ of points in $\IR^d$ satisfying Condition~\ref{cond:delta}.
Let $\tau >0$ be fixed.
For $n \in \IN$, let $(A_{\theta})_{\theta \in \Theta}$
and
$(B_{\theta})_{\theta \in \Theta}$ be families of $n \times n$ matrices. Assume that for all $n \in \IN$, $i,j=1,\ldots,n$ and $\theta \in \Theta$,
\[
\left|
(A_{\theta})_{i,j}
\right|
\leq
\frac{\Csup}{ 1 + |x_i - x_j|^{d + \tau} }
\; \; \;
 \mbox{and}
 \; \; \;
 \left|
(B_{\theta})_{i,j}
\right|
\leq
\frac{\Csup}{ 1 + |x_i - x_j|^{d + \tau} }
\]
 where $\Csup$ does not depend on $n,i,j,\theta$. 
Then we have
for all $n \in \IN$, $i,j=1,\ldots,n$ and $\theta \in \Theta$,
\[
\left|
(A_{\theta} B_{\theta})_{i,j}
\right|
\leq
\frac{\Csup}{ 1 + |x_i - x_j|^{d + \tau} }
\]
 where $\Csup$ does not depend on $n,i,j,\theta$. 
\end{lemma}

\begin{proof}
We have, 
\begin{align*}
\left|
( A_{\theta} B_{\theta} )_{i,j}
\right|
&= 
\left|
\sum_{\ell=1}^n
 (A_\theta)_{i,\ell} (B_\theta)_{\ell,j}
\right|
\\
&\leq 
\sum_{\ell = 1}^n
\frac{\Csup}{ 1 + |x_i - x_\ell|^{d + \tau} }
\frac{\Csup}{ 1 + |x_j - x_\ell|^{d + \tau} }
\\
&\leq 
\underset{
\substack{
\ell = 1,\ldots,n
\\
|x_i - x_{\ell}|
\leq 
|x_j - x_{\ell}|
}
}
{\sum}
\frac{\Csup}{ 1 + |x_i - x_\ell|^{d + \tau} }
\frac{\Csup}{ 1 + (|x_i - x_j|/2)^{d + \tau} }
\\
&\quad +
\underset{
\substack{
\ell = 1,\ldots,n
\\
|x_j - x_{\ell}|
\leq 
|x_i - x_{\ell}|
}
}
{\sum}
\frac{\Csup}{ 1 + (|x_i - x_j|/2)^{d + \tau} }
\frac{\Csup}{ 1 + |x_j - x_\ell|^{d + \tau} }
\\
&\leq 
\Csup
\frac{\Csup}{ 1 + |x_i - x_j|^{d + \tau} }
\max_{a=1,\ldots,n}
\sum_{b=1}^n
\frac{1}{ 1 + |x_a - x_b|^{d + \tau} }
 \\& 
 \leq 
 \frac{\Csup}{ 1 + |x_i - x_j|^{d + \tau} }
\end{align*}
from Lemma 4 in \cite{furrer2016asymptotic}.
\end{proof}

\begin{lemma} \label{lem:Rtheta:Rtheta:m1}
Consider the setting of Section \ref{sub:sec:Framework:estitheta}.
Under Conditions~\ref{cond:cov:Y:theta} and~\ref{cond:cov:Y:theta:fourier}, we have
\begin{equation} \label{eq:lambda1:Rthetam1}
\sup_{\theta \in \Theta}
\lambda_1 ( R_{\theta}^{-1} )
\leq 
\Csup,
\end{equation}
\begin{equation} \label{eq:lambda1:Rtheta}
\sup_{\theta \in \Theta}
\lambda_1 ( R_{\theta} )
\leq 
\Csup,
\end{equation}
and
\begin{equation}  \label{eq:lambda1:dTheta}
\sup_{ \substack{
\theta \in \Theta
\\
\ell = 1,2,3
\\
i_1,\ldots,i_\ell = 1,\ldots,p
}}
\lambda_1
\Bigl(
\frac{\partial R_{\theta}}{ \partial \theta_{i_1} \ldots \partial \theta_{i_\ell} }
\Bigr)
\leq \Csup.
\end{equation}
\end{lemma}

\begin{proof}
Conditions~\ref{cond:delta}, \ref{cond:cov:Y:theta} and \ref{cond:cov:Y:theta:fourier} imply \eqref{eq:lambda1:Rthetam1} from Theorem 5 in \cite{bachoc2016smallest}. Conditions~\ref{cond:delta} and \ref{cond:cov:Y:theta} imply \eqref{eq:lambda1:Rtheta} and \eqref{eq:lambda1:dTheta} from Lemma~6 in \cite{furrer2016asymptotic}.
\end{proof}

\begin{lemma} \label{lem:Rtheta:m1:ij}
Consider the setting of Section \ref{sub:sec:Framework:estitheta}.
Under Conditions~\ref{cond:cov:Y:theta} and \ref{cond:cov:Y:theta:fourier}, we have, for $n \in \mathbb{N}$ and $i,j \in \{1 , \ldots , n\}$,
\[
\sup_{\theta \in \Theta}
\left|
(R_{\theta}^{-1})_{i,j}
\right|
\leq 
\frac{\Csup }{1+|s_i - s_j|^{d + \Cinf}},
\]
where $\Csup$ and $\Cinf$ do not depend on $n,i,j,\theta$.
\end{lemma}

\begin{proof}
One can show that the proof of Theorem~\ref{theorem:decay:coeff:Rinverse} can be made uniform over $\theta \in \Theta$, thus yielding Lemma~\ref{lem:Rtheta:m1:ij}. 
\end{proof}

\begin{lemma} \label{lem:diag:Rtheta:m1:ij}
Consider the setting of Section \ref{sub:sec:Framework:estitheta}.
Under Conditions~\ref{cond:cov:Y:theta} and \ref{cond:cov:Y:theta:fourier}, we have,
\begin{align}
& \inf_{\theta \in \Theta} \lambda_n(R_{\theta})\label{eq:lamN:R}
\geq \Cinf,\\
& \inf_{\theta \in \Theta} \lambda_n(\diag(R_{\theta}^{-1}))
\geq \Cinf.\label{eq:lamN:R:inv}
\end{align}
\end{lemma}

\begin{proof}
Equation \eqref{eq:lamN:R} holds from \eqref{eq:lambda1:Rthetam1}. Then, \eqref{eq:lamN:R:inv} follows from \eqref{eq:lamN:R} as in Lemma D.6 in \cite{bachoc14asymptotic}.
\end{proof}

\subsection{Proofs of the main results}

\begin{proof}[{\bfseries Proof of Lemma~\ref{lemma:sub:exp:ass:pos}}]
As a special case of Lemma~\ref{lemma:cov:muliple:products}, $k'$ satisfies Condition~\ref{cond:sub:exp:ass:pos} i).

Let us now show that $k'$ satisfies Condition~\ref{cond:sub:exp:ass:pos} ii).
 Let $(x_i)$ satisfy Condition~\ref{cond:delta}.
Let $n \in \IN$ be fixed and let $R$ be the $n \times n$ covariance matrix $k(x_i-x_j)_{i,j=1,\dots,n}$.
Let $a_1,\ldots,a_n \in \IR$. We have
\begin{align*}
\sum_{i,j=1}^n a_i a_j R_{i,j}
& =
\Var \Bigl( \sum_{i=1}^n a_i F( X(x_i) ) \Bigr).
\end{align*}
We now let $z = (X(x_1),\ldots,X(x_n))^\top$ and $g:\IR^n \to \IR$ be defined by $g(t) = \sum_{i=1}^n a_i F(t_i)$. The gradient of $g$ at $t$ is $ \nabla g(t) =  ( a_1 F'(t_1) , \ldots , a_n F'(t_n) )^\top$. We use the inequality in Theorem 3.7 in \cite{cacoullos1982upper}. This yields
\[
\sum_{i,j=1}^n a_i a_j R_{i,j}
\geq
\IE\left[\nabla g(z) \right]^\top \Cov( z ) \IE\left[\nabla g(z) \right].
\]
From Condition~\ref{cond:sub:exp:ass:pos} ii), we have $\lambda_1(\Cov( z )) \geq \Cinf$. This yields
\begin{align*}
\sum_{i,j=1}^n a_i a_j R_{i,j}
& \geq
\Cinf
\sum_{i=1}^n
\IE^2 \left[ \left( \nabla g(z) \right)_i \right]
\\
& =
\Cinf
\sum_{i=1}^n
a_i^2 \IE[ F'(z_i) ]^2
\\
& =
\Cinf
\Bigl(
\sum_{i=1}^n
a_i^2 
\Bigr)
\IE[ F'(z_1) ]^2.
\end{align*}
From Condition~\ref{cond:T}, the above expectation is non-zero, which concludes the proof.
\end{proof}

\begin{proof}[{\bfseries Proof of Lemma~\ref{lemma:lambda:inf:R}}]

The fact that $k$ satisfies Condition~\ref{cond:sub:exp:ass:pos} ii) follows from Theorem 4 in \cite{bachoc2016smallest}.

Let us now consider the case where $F$ is defined by $F(x) = x^{2r} + u$.
Since we consider a covariance function we can assume that $u = 0$  without loss of generality.
Assume also that $k(0) = 1$ without loss of generality.
From Lemma~\ref{lemma:cov:muliple:products}, $k'$ satisfies Condition~\ref{cond:sub:exp:ass:pos} i). Let us show that Condition~\ref{cond:sub:exp:ass:pos} ii) is satisfied.
 Let $a,b \in \mathbb{R}^d$, let $c = k(a-b)$ and let $\lambda = (1-c^2)^{1/2}$. With $(A_1,A_2) \sim \cN[0,I_2]$, we have
\begin{align*}
k'(a-b)
& =
\Cov( A_1^{2r} , (cA_1 + \lambda A_2)^{2r} )
\\
& =
\Cov \Bigl( A_1^{2r} , \sum_{i=0}^{2r} {{2r}\choose{i}} c^i  \lambda^{2r - i} A_1^i A_2^{2r - i}  \Bigr)
\\
& =
\sum_{i=0}^{2r}
{{2r}\choose{i}} c^{i}  \lambda^{2r - i}
\Cov
 \bigl( 
 A_1^{2r},
A_1^i A_2^{2r - i}
\bigr).
\end{align*}
By independence of $A_1$ and $A_2$, we obtain, for $i=0,\ldots,2r$,
\begin{align} \label{eq:in:proof:lambda:inf:cov}
\Cov
  \bigl( 
 A_1^{2r},
A_1^i A_2^{2r - i}
\bigr)
& =
\IE[A_2^{2r-i}]
\big(
\IE[A_1^{2r+i}]
-
\IE[A_1^{2r}]
\IE[A_1^{i}]
\big).
\end{align}
From Isserlis' theorem, one can show that \eqref{eq:in:proof:lambda:inf:cov} is zero if $i$ is odd and is strictly positive if $i$ is even. As a consequence, we have 
\[
k'(a-b) = \sum_{i=0}^{r}
\alpha_i k(a-b)^{2i}
\]
with $\alpha_1,\ldots,\alpha_r >0$. Hence, the Fourier transform of $k'$ is a linear combination of multiple convolutions of the Fourier transform of $k$, with strictly positive components. Since the Fourier transform of $k$ is strictly positive everywhere, then also the Fourier transform of $k'$ is strictly positive everywhere. Hence, from Theorem 4 in \cite{bachoc2016smallest}, $k'$ satisfies Condition~\ref{cond:sub:exp:ass:pos}~ii).
\end{proof}

\begin{proof}[{\bfseries Proof of Theorem~\ref{theorem:decay:coeff:Rinverse}}]
Condition~1 and Lemma~6 in \cite{furrer2016asymptotic} imply that the spectral norms of $R^{-1}$ and $R$ are bounded functions of $n$. Let $\Csup= \sup_n \lambda_1(R) < \infty $ and $\Cinf= \inf_n \lambda_n(R) >0$.

We write
\begin{align} \label{eq:series:inverse}
  R^{-1}=\frac{1}{\Csup}\Bigl(I-\bigl(I-\frac{R}{\Csup}\bigr)\Bigr)^{-1}=\frac{1}{\Csup}\sum_{\ell=0}^\infty \Bigl(I-\frac{R}{\Csup}\Bigr)^{\ell}.
\end{align}
We remark that the above sum is well-defined because the eigenvalues of $I - R/\Csup$ are between $0$ and $1-\Cinf/\Csup$.

We denote $M=I-\Csup^{-1}R$ and $h_{i,j} = |s_i - s_j|$.
Let $1 \leq A < \infty$ and $a >0$ be fixed such that $M_{i,j} \leq A e^{- 2 a h_{i,j}}$.
Let $\delta =  \inf_{i , j \in \mathbb{N}, i \neq j} |s_i -s_j| >0$ (Condition~\ref{cond:delta}).
Let $D< \infty$ be a constant such that $D \geq 1$ and, for any $L \geq \delta$ and $i \in \mathbb{N}$, the set $\{ j \in \mathbb{N}; |s_i - s_j|  \leq L \}$ has no more than $(D/2) L^d$ elements. 

Let $ 0 < \mu < \infty$ be fixed. We show by induction over $\ell \in \mathbb{N}$ that there exists a constant $1 \leq \varphi < \infty $, depending on $\mu$ but not depending on $\ell,i,j$, such that for $\ell \leq \mu \log( h_{i,j} )$,
\begin{align} 
| (M^\ell)_{i,j} | \leq 
A^\ell \varphi^\ell D^\ell h_{i,j}^{d\ell - d} e^{ - a h_{i,j} }.
 \label{eq:ind} 
\end{align}
In the case $\log( h_{i,j} ) < 0$, there is nothing to prove in \eqref{eq:ind}, so we consider $i,j$ such that $\log( h_{i,j} ) \geq 0$ when proving \eqref{eq:ind}. 

For $\ell=1$, \eqref{eq:ind} holds. Assume that \eqref{eq:ind} holds for some $\ell \in \mathbb{N}$. We have
\begin{align*}
| (M^{\ell+1})_{i,j} | 
& =
\sum_{r = 1}^n 
(M^\ell)_{i,r} M_{j,r} 
\\
& \leq 
\sum_{r = 1}^n 
A^\ell \varphi^\ell D^\ell h_{i,r}^{d\ell - d} e^{ - a h_{i,r} }  
A e^{ - 2 a h_{r,j}} 
\\
& =
A^{\ell+1} \varphi^\ell D^\ell 
\sum_{r = 1}^n 
 h_{i,r}^{d\ell - d} e^{ - a h_{i,r} } 
 e^{ - 2 a h_{r,j}}. 
\end{align*}
Let now $B_i = \{ x \in \mathbb{R}^d ; |x - s_i| \leq  |s_i - s_j|\}$ and $B_j = \{ x \in \mathbb{R}^d ; |x - s_j| \leq  |s_i - s_j|\}$. From the triangle inequality we obtain
\begin{align*}
| (M^{\ell+1})_{i,j} | 
& \leq 
A^{\ell+1} \varphi^\ell D^\ell 
\sum_{ r \in \mathbb{N} ; s_r \in B_i \cup B_j}
h_{i,r}^{d\ell - d} e^{ -  a h_{i,r} } 
 e^{ - 2 a h_{r,j}}
 \\
 & \quad
 +
 A^{\ell+1} \varphi^\ell D^\ell 
\sum_{ r \in \mathbb{N} ; s_r \in B_i^c \cap B_j^c}
h_{i,r}^{d\ell - d} e^{ -  a h_{i,r} } 
 e^{ - 2 a h_{r,j}} 
 \\
 & \leq 
 A^{\ell+1} \varphi^\ell D^\ell 
 2 (D/2) h_{i,j}^d
h_{i,j}^{d\ell-d} 
  e^{ - a h_{i,j} } 
  \\
  & \quad +  
   A^{\ell+1} \varphi^\ell D^\ell    
   e^{-2 a h_{i,j}}
   \sum_{ r \in \mathbb{N} ; s_r \in B_i^c }
   h_{i,r}^{d\ell-d} e^{- a h_{i,r}}
   \\
   & \leq 
   A^{\ell+1} \varphi^\ell 
 D^{\ell+1}
h_{i,j}^{d(\ell+1)-d} 
  e^{ -  a h_{i,j} } 
  \\
  & \quad
  +  
   A^{\ell+1} \varphi^\ell D^\ell    
   e^{ -  a h_{i,j} } 
   \Bigl(
    e^{ -  a h_{i,j} } 
    \sum_{r \in \mathbb{N}}
    Q 
    r^{d \mu \log( h_{i,j} )  }
    e^{ - a (r-1) }
   \Bigr),
\end{align*}
where for the last inequality we let $Q r^d$ be an upper bound on the cardinality of $\{ b \in \mathbb{N} ; |s_b - s_i| \in [ r -1, r ] \}$ for all $i \in \mathbb{N}$. The constant $Q$ is finite and depends only on $d$ and $\delta$ from Condition~\ref{cond:delta}. We also let $\ell \leq \mu \log( h_{i,j} )$ to show the last above inequality.
Hence, in order to finish the proof of \eqref{eq:ind}, it remains to show that the term $(.)$ in the above display is a bounded function of $h_{i,j}$, and to let $\varphi /2 \geq 1$ be a bound for the term $(.)$. 

We have, for $h_{i,j}$ large enough, with $\lceil \cdot \rceil$ the integer ceiling,
\begin{align*}
\Bigl(
    e^{ -  a h_{i,j} } 
    \sum_{r \in \mathbb{N}}
    Q 
    r^{d \mu \log( h_{i,j} )  }
    e^{ - a (r-1) }
   \Bigr)
   & =
Q e^a e^{-a h_{i,j}}   
\sum_{r \in \IN}
r^{ \lceil d \mu \log( h_{i,j} ) \rceil }
e^{-ar}   
\\
&
=
Q e^a e^{-a h_{i,j}}   
\left( 
\frac{2}{a}
\right)^{
\lceil d \mu \log( h_{i,j} ) \rceil
}
\sum_{r \in \IN}
\left(
\frac{ar}{2}
\right)^{
 \lceil d \mu \log( h_{i,j} ) \rceil 
 }
e^{-ar}   
\\
& \leq 
Q e^a e^{-a h_{i,j}}   
\left( 
\frac{2}{a}
\right)^{
\lceil d \mu \log( h_{i,j} ) \rceil
}
\lceil d \mu \log( h_{i,j} ) \rceil!
\sum_{r \in \IN}
e^{ar/2}   
e^{-ar}   
\\
& =
Q e^a e^{-a h_{i,j}}   
\left( 
\frac{2}{a}
\right)^{
\lceil d \mu \log( h_{i,j} ) \rceil
}
\lceil d \mu \log( h_{i,j} ) \rceil!
\frac{1}{1 - e^{-ar/2}}.
\end{align*}
The above function of $h_{i,j}$ clearly goes to $0$ as $h_{i,j}$ goes to $\infty$. Thus, the above term $(.)$ is bounded and thus \eqref{eq:ind} is proved.

Coming back to \eqref{eq:series:inverse}, using \eqref{eq:ind} and using the triangle inequality, we obtain, letting $\Delta = 1 - \Cinf/\Csup$, and for $h_{i,j}$ large enough,
\begin{align} \label{eq:concluding:coeff:inverse}
\left| 
(R^{-1})_{i,j}
\right|
&
\leq 
\Bigl(
\sum_{1 \leq \ell \leq \mu \log(h_{i,j})}
A^\ell \varphi^\ell D^\ell h_{i,j}^{d\ell - d} e^{ - a h_{i,j} }
\Bigr)
+ \sum_{\mu \log(h_{i,j}) \leq \ell \leq \infty}
\Delta^\ell
\notag
\\
& \leq 
\mu \log( h_{i,j} )
(A \varphi D)^{\mu \log( h_{i,j} )}
h_{i,j}^{ d \mu \log( h_{i,j} ) }
e^{ -a h_{i,j} }
+
\frac{\Delta^{\mu \log( h_{i,j} )}}{ 1 - \Delta }
\notag
\\
& = 
\mu \log( h_{i,j} )
(A \varphi D)^{\mu \log( h_{i,j} )}
h_{i,j}^{ d \mu \log( h_{i,j} ) }
e^{ -a h_{i,j} }
+
\frac{h_{i,j}^{\mu \log( \Delta )}}{ 1 - \Delta }.
\end{align}
In the above display, for any $\tau < \infty$ in the statement of Theorem~\ref{theorem:decay:coeff:Rinverse}, we can choose $\mu$ such that $\mu \log( \Delta ) \leq - d - \tau$. Then, it is clear that the first summand in \eqref{eq:concluding:coeff:inverse} is also smaller than a constant (depending on $\tau$) time $h_{i,j}^{ -d-\tau }$. This concludes the proof of Theorem~\ref{theorem:decay:coeff:Rinverse}, since also $\sup_{n \in \mathbb{N}} \max_{i,j=1,\ldots,n} |(R^{-1})_{i,j}|$ is bounded by $\Csup$.
\end{proof}

\begin{proof}[{\bfseries Proof of Theorem~\ref{theorem:generic:TCL}}]

We have
\begin{align}
n \Var (V_n)
=
\frac{1}{n}
\sum_{i,j,k,l=1}^n
A_{i,j} A_{k,l}
\Cov( y_iy_j,y_ky_l ).\label{eq:nVarVn}
\end{align}
From Lemma~\ref{lemma:covijkl}, we obtain 
\begin{align*}
n \Var (V_n)
& \leq
\frac{1}{n}
\Csup
\sum_{i,j=1}^n
|A_{i,j}| \\
& 
\leq
\Csup
\max_{i=1,\ldots,n}
\sum_{j=1}^n
\frac{ 1 }{ 1 + | s_i - s_j |^{d + \Cinf }}
\\
&
\leq \Csup
\end{align*}
from Lemma 4 in \cite{furrer2016asymptotic}. Hence, $n \Var( V_n )$ is bounded as $n \to \infty$.

Assume now that 
\begin{equation} \label{eq:in:proof:TCL:contradiction}
d_w
\left(
\mathcal{L}_n
,
\cN
\left[
0,
n \var(V_n)
\right]
\right)
\not 
\to 
0
\end{equation}
as $n \to \infty$. Because $n \Var(V_n)$ is bounded, there exists a subsequence $\phi(n)$ such that
\begin{equation} \label{eq:in:proof:TCL:contradiction:2}
d_w
\left(
\mathcal{L}_{\phi(n)}
,
\cN
\left[
0,
\phi(n) \var(V_{\phi(n)})
\right]
\right)
\not 
\to 
0
\end{equation}
as $n \to \infty$
and $ \phi(n) \Var(V_{\phi(n)}) \to V \in [0,\infty)$ as $n \to \infty$. It is then simple to show that this implies 
\begin{equation} \label{eq:in:proof:TCL:contradiction:3}
d_w
\left(
\mathcal{L}_{\phi(n)}
,
\cN
\left[
0,
V
\right]
\right)
\not 
\to 
0
\end{equation}
as $n \to \infty$. If $V = 0$, then, from Chebyshev inequality, $\mathcal{L}_{\phi(n)}$ converges to a Dirac mass at zero and so \eqref{eq:in:proof:TCL:contradiction:3} does not hold, yielding a contradiction.

Hence it remains to consider the case $ \phi(n) \Var(V_{\phi(n)}) \to V \in (0,\infty)$ as $n \to \infty$ and where \eqref{eq:in:proof:TCL:contradiction:3} holds.

To reach a contradiction, we will show that 
\[
\sqrt{ \phi(n) } \left( \frac{ V_{\phi(n)} - \IE[V_{\phi(n)}]}{ \sqrt{V} } \right)
\to_{ n \to \infty  }
\cN[ 0 , 1].
\]

To simplify notations in the sequel, without loss of generality, we will consider that $\phi(n) = n$ and show that
\begin{equation} \label{eq:TCL:to:show}
\sqrt{n } \left( \frac{ V_n - \IE[V_{n}]}{ \sqrt{V} } \right)
\to_{ n \to \infty  }
\cN[ 0 , 1],
\end{equation}
where $n \Var(V_n) \to V \in (0,\infty)$ as $n \to \infty$.
From Slutsky's lemma it is sufficient to show that
\begin{equation} \label{eq:TCL:to:show:2}
\sqrt{n } \left( \frac{ V_n - \IE[V_{n}]}{ \sqrt{n \Var(V_n)} } \right)
\to_{ n \to \infty  }
\cN[ 0 , 1].
\end{equation}

For $K \geq 0$, let 
\[
V^{(K)}_n = \frac{1}{n} y^\top A^{(K)} y,
\]
with the notation Lemma~\ref{lemma:A:sparse:approximation}.
We have
\begin{flalign*}
&
\sup_{n \in \mathbb{N}}
\IE
\biggl[
\biggl(
\sqrt{n } \biggl( \frac{ V_n - \IE[V_{n}]}{ \sqrt{n \Var(V_n)} } \biggr)
-
\sqrt{n } \biggl( \frac{ V^{(K)}_n - \IE[V^{(K)}_{n}]}{ \sqrt{n \Var(V^{(K)}_n)} } \biggr)
\biggr)^2
\bigg]
&
\\
& \leq 
2
\sup_{n \in \mathbb{N}}
\Var
\biggl(
\sqrt{n } \biggl( \frac{ V_n - V^{(K)}_n )}{ \sqrt{n \Var(V_n)} } \biggr)
\biggr)
+
2
\sup_{n \in \mathbb{N}}
\Var
\biggl(
\sqrt{n } V^{(K)}_n
\biggl(
\frac{ 1}{ \sqrt{n \Var(V_n)} } 
-
 \frac{ 1}{ \sqrt{n \Var(V^{(K)}_n)} } 
 \biggr)
\biggr)
&
\\
&
\to_{n \to \infty} 0
&
\end{flalign*}
from Lemma~\ref{lemma:A:sparse:approximation} and because $V >0$. Hence, from Theorem 4.2 in \cite{billingsley68convergence} (as in \cite{Neumann13}), it is sufficient to show that there exists $L \in (0,\infty)$ such that for any fixed $K \geq L$, we have
\[
\sqrt{n } \left( \frac{ V^{(K)}_n - \IE[V^{(K)}_{n}]}{ \sqrt{n \Var(V^{(K)}_n)} } \right)
\to^{\mathcal{L}}
\cN[ 0 , 1]
\]
as $n \to \infty$, in order to prove \eqref{eq:TCL:to:show:2} and thus to conclude the proof.
We remark that, because of Lemma~\ref{lemma:A:sparse:approximation}, we have $\big| \liminf_{n \to \infty} n \Var(V^{(K)}_n) - \liminf_{n \to \infty} n \Var(V^{}_n)\big|$ goes to $0$ as $K \to \infty$. Hence, we may take $L$ such that $ \liminf_{n \to \infty} n \Var(V^{(K)}_n) >0$ for $K \geq L$. Hence, up to extracting a subsequence, it is sufficient to show
\begin{equation} \label{eq:TCL:to:show:3}
\sqrt{n } \bigl( V^{(K)}_n - \IE[V^{(K)}_{n}] \bigr)
\to^{\mathcal{L}}
\cN[ 0 , V^{(K)}],
\end{equation}
where $ n \Var( V^{(K)}_n ) \to V^{(K)} >0 $ as $n \to \infty$.
We have
\begin{align*}
 \sqrt{n } \bigl( V^{(K)}_n - \IE[V^{(K)}_{n}] \bigr)
& =
\frac{1}{\sqrt{n}}
\sum_{i,j=1}^n
\left(
y_i
y_j
- \IE[ y_i y_j ]
\right)
A_{i,j}
1_{ |s_i - s_j| \leq K }
\\
& =
\frac{1}{\sqrt{n}}
\sum_{i=1}^n
\Bigl(
\sum_{ \substack{ j=1,\ldots,n \\ |s_i - s_j| \leq K } }
A_{i,j}
\bigl(
T(Z(s_i))
T(Z(s_j))
-
\IE
\left[
T(Z(s_i))
T(Z(s_j))
\right]
\bigr)
\Bigr)
\\
& = 
\frac{1}{\sqrt{n}}
\sum_{i=1}^n
X_{n}(s_i),
\end{align*}
say, where $X_n$ can be interpreted as a centered random field defined on  
$ (s_i)_{i \in \mathbb{N}} $. We will now show that the sequence of random fields $(X_n)_{n \in \mathbb{N}}$ satisfies the conditions of Corollary~1 of \cite{jenish2009central}.

We let, for $k,l \in \mathbb{N}$ and $r \geq 0$ 
\begin{align*}
\bar{\alpha}_{k,l}(r)&
=
\sup_{n \in \mathbb{N}}\;
\sup 
\Big\{
\left|
\IP( A \cap B )
-
\IP(A)
\IP(B)
\right|
;\\
&\quad A \in \sigma( X_n( s_{I_1} ) ,\ldots , X_n( s_{I_{\bar{k}}} ) ),
B \in \sigma( X_n( s_{J_1} ) ,\ldots ,  X_n( s_{J_{\bar{l}}} ) ),
\\
&\quad
\bar{k} \leq k,
\bar{l} \leq l,
I_1,\ldots,I_{\bar{k}},J_1,\ldots,J_{\bar{k}} \in \{1,\ldots,n\},
\min_{\tilde{k} = 1,\ldots,\bar{k},\tilde{l} = 1,\ldots,\bar{l}}
| s_{I_{\tilde{k}}} - s_{J_{\tilde{l}}} |
\geq r
\Big\}.
\end{align*}
Let $N_K = \sup_{n \in \mathbb{N}} \max_{ i=1,\ldots,n } \sum_{j=1}^n 1_{ |s_j - s_i| \leq K } $.
Then $N_K \leq \Csup$, where $\Csup$ depends only on $K$, $d$ and $\delta$ from Condition~\ref{cond:delta}. 
We remark that for $i=1,\ldots,n$, $X_n(s_i)$ is a function of the variables $Z(s_{I(n,i,1)}),\ldots,Z(s_{I(n,i,\gamma(n,i))})$, with $\gamma(n,i) \leq N_K$ and with $|s_{I(n,i,\gamma)} - s_i | \leq K$ for $\gamma=1,\ldots,\gamma(n,i)$. Furthermore, for $|s_i - s_j| \geq r$ we have for $\gamma_i = 1,\ldots,\gamma(n,i)$ and for $\gamma_j = 1,\ldots,\gamma(n,j)$ that $| s_{I(n,i,\gamma_i)} - s_{I(n,j,\gamma_j)} | \geq r - 2 K$. Hence, form Lemma~\ref{lemma:alpha:mixing}, we have, for any $k \in \mathbb{N}$
\[
\sup_{l \in \mathbb{N}}
\bar{\alpha}_{k,l}(r)
\leq 
\Csup e^{ -\Cinf r },
\]
where $\Csup$ and $\Cinf$ may depend on $k$ and $K$.

We now let $D = ( s_i )_{i \in \mathbb{N}}$, $D_n = (s_1,\ldots,s_n)$, $Z_{i,n} = n^{-1/2} X_n(s_i)$ for $n \in \mathbb{N}$ and $i=1,\ldots,n$. We also let $c_{i,n} = n^{-1/2} $ for $n \in \mathbb{N}$ and $i=1,\ldots,n$. We remark that $Z_{i,n} / c_{i,n}$ can be written as $f( w )$ where $w$ is a Gaussian vector of dimension less than $N_K$, with variances $1$ and where $|f(x)| \leq \Csup e^{ \Csup |x|} $, where $\Csup$ does not depend on $n \in \mathbb{N}$ and $i=1,\ldots,n$. One can thus show, from the Cauchy-Schwarz inequality and with the same techniques as in Lemma~\ref{lemma:finite:mean:function:gaussian:vector}, that for any $q>0$
\begin{equation} \label{eq:in:TCL:lindeberg:like}
\lim_{M \to + \infty}
\sup_{n \in \mathbb{N}}
\max_{i=1,\ldots,n}
\IE
\left[
| Z_{i,n} / c_{i,n} |^{2+q}
1_{ | Z_{i,n} / c_{i,n} | \geq M }
\right]
= 0.
\end{equation}
With the previous notation and with \eqref{eq:in:TCL:lindeberg:like}, one can show that all the assumptions of Corollary 1 in \cite{jenish2009central} are satisfied. This shows \eqref{eq:TCL:to:show:3}
and thus concludes the proof.
\end{proof}

\begin{proof}[{\bfseries Proof of Theorem~\ref{theorem:consistency:thetaML}}]
Let $\theta \in \Theta$ be fixed. We have
\[
    \Var( L_{\theta} ) = \frac{1}{n} n \Var \left( \frac{1}{n} \left(y^\top R_{\theta}^{-1} y \right) \right).
\]

From Lemma~\ref{lem:Rtheta:m1:ij} and
Theorem~\ref{theorem:generic:TCL}, applied with $A_n = R_{\theta}^{-1}$, we obtain $\Var( L_{\theta} ) \to 0$ as $n \to \infty$.
For $i=1,\ldots,p$,
\[
   \frac{\partial L_{\theta}}{ \partial \theta_{i}} = \frac{1}{n} \trace\left( R_{\theta}^{-1} \frac{\partial R_{\theta}}{\partial \theta_{i}} \right) + \frac{1}{n} \left(y^\top \left( -R_{\theta}^{-1} \frac{\partial R_{\theta}}{\partial \theta_{i}}R_{\theta}^{-1} \right) y \right),
\]
which can be rewritten for convenience as
\[
 \frac{\partial L_{\theta}}{ \partial \theta_{i}} = \frac{1}{n} \trace\left( P_{\theta,i} \right) + \frac{1}{n} \left(y^\top Q_{\theta,i} y \right)
\]
with
\[
P_{\theta,i} = R_{\theta}^{-1} \frac{\partial R_{\theta}}{\partial \theta_{i}}
\; \; \;
 \mbox{and}
 \; \; \;
 Q_{\theta,i} = -R_{\theta}^{-1} \frac{\partial R_{\theta}}{\partial \theta_{i}}R_{\theta}^{-1}.
\]

The matrices $R_{\theta}^{-1}$ and ${\partial R_{\theta}}/{\partial \theta_{i}}$ are both valid choices for $A_\theta$ and $B_\theta$ in Lemma~\ref{lem:product:matrix}. From Gerschgorin Circle Theorem (GCT) and Lemma 4 in \cite{furrer2016asymptotic}, we obtain $\sup_{\theta \in \Theta} \lambda_1(P^\top_{\theta,i}P_{\theta,i}) \leq \Csup$ and $\sup_{\theta \in \Theta} \lambda_1(Q_{\theta,i}) \leq \Csup$. This, in turn implies that $\sup_{\theta \in \Theta} \rho_1(P_{\theta,i}) \leq \Csup$. It follows that

\begin{align*}
\max_{i=1,\ldots,n} \sup_{\theta \in \Theta}
\left|
\frac{\partial L_{\theta}}{ \partial \theta}
\right|& \leq \sup_{\theta \in \Theta}
\left| \frac{1}{n} \trace\left( P_{\theta,i} \right) + \frac{1}{n} \left(y^\top Q_{\theta,i} y \right) \right|\\
&\leq \Csup + \Csup \frac{||y||^2}{n}\\
&= O_p(1).
\end{align*}

Hence, Theorem~\ref{theorem:consistency:thetaML} can be proved by proceeding as in the proof of Proposition 3.1 in \cite{bachoc14asymptotic}.
\end{proof}

\begin{proof}[{\bfseries Proof of Theorem~\ref{theorem:CLT:thetaML}}]
From the proof of Theorem~\ref{theorem:consistency:thetaML}, we have for $i=1,\ldots,p$
\[
    \partial L_{\theta} / \partial \theta_i  = \frac{1}{n} \trace( P_{\theta,i} ) + \frac{1}{n} (y^\top Q_{\theta,i} y),
\]
where $P_{\theta,i}$ is a $n \times n$ matrix satisfying $\sup_{\theta \in \Theta}
| (P_{\theta,i})_{a,b} | \leq \Csup/(1+|s_a-s_b|^{d+\Cinf})$ and $Q_{\theta,i}$ is a $n \times n$ symmetric matrix satisfying $\sup_{\theta \in \Theta}
| (Q_{\theta,i})_{a,b} | \leq \Csup/(1+|s_a-s_b|^{d+\Cinf})$.

One can check that $\partial L_{\theta_0} / \partial \theta_i $ has mean zero for $i=1,\ldots,p$, since the mean value of $\partial L_{\theta_0} / \partial \theta_i $ is calculated as if $Y$ were a Gaussian process with zero-mean and covariance function $k_{Y,\theta_0}$.
Let $\partial L_{\theta_0} / \partial \theta$ be the gradient column vector of $L_{\theta}$ at $\theta_0$.
From Theorem~\ref{theorem:generic:TCL}, with $\mathcal{L}_{\Sigma,\theta_0,n}$ the distribution of $\sqrt{n} \partial L_{\theta_0} / \partial \theta$, as $n \to \infty$,

\begin{equation} \label{eq:for:TCL:theta:1:ML}
d_w ( \mathcal{L}_{\Sigma,\theta_0,n} , \cN[ 0 , \Sigma_{\theta_0} ] ) \to 0.
\end{equation}

In addition, for $i \in \{ 1 , \ldots , p\}$, the sequence $(n \Var(\partial L_{\theta_0} / \partial \theta_i))$ is bounded, which implies that the elements of $\Sigma_{\theta_0}$ are bounded too.

One can check that the mean value of $\partial L_{\theta_0} / \partial \theta_i \partial \theta_j$ is $(M_{\theta_0})_{i,j}$ (also because this mean value is calculated as if $Y$ were a Gaussian process with zero-mean and covariance function $k_{Y,\theta_0}$).

Also, for $i,j=1,\ldots,p$, we have
\[
\partial L_{\theta} / \partial \theta_i 
\partial \theta_j 
 = \frac{1}{n} \trace( C_{\theta,i,j} ) + \frac{1}{n} (y^\top D_{\theta,i,j} y),
\]
where $C_{\theta,i,j} $ and $D_{\theta,i,j}$ are sums of products of the matrices $R_{\theta}^{-1}$, $R_{\theta}$ and the first and second derivative matrices of $R_{\theta}$ (see e.g., \cite{bachoc14asymptotic}). Hence, from Condition~\ref{cond:cov:Y:theta} and Lemma~\ref{lem:Rtheta:m1:ij} used inside Lemma~\ref{lem:product:matrix}, we have $\sup_{\theta \in \Theta} | (C_{\theta,i,j})_{a,b} | \leq \Csup/(1+|s_a-s_b|^{d+\Cinf})$ and $\sup_{\theta \in \Theta} | (D_{\theta,i,j})_{a,b} | \leq \Csup/(1+|s_a-s_b|^{d+\Cinf})$.

Thus, the variance of $\partial L_{\theta_0} / \partial \theta_i \partial \theta_j$ goes to zero as $n \to \infty$ from Theorem~\ref{theorem:generic:TCL}. Hence 
\begin{equation} \label{eq:for:TCL:theta:2:ML}
\partial L_{\theta_0} / \partial \theta_i \partial \theta_j
\to^p 
(M_{\theta_0})_{i,j}
\end{equation}
as $n \to \infty$.

It can be shown, similarly as in the proof of Proposition 3.3 in \cite{bachoc14asymptotic} that
\begin{equation} \label{eq:for:TCL:theta:3:ML}
\liminf_{n \to \infty}
\lambda_p ( M_{\theta_0} ) >0.
\end{equation}
Hence, \eqref{eq:in:TCL:theta:lambda1} follows.

Then, for $i,j,\ell \in \{1 , \ldots , p\}$, we have
\[
    \partial L_{\theta} / \partial \theta_i \partial \theta_j  \partial \theta_\ell  = \frac{1}{n} \trace( E_{\theta,i,j,\ell} ) + \frac{1}{n} (y^\top F_{\theta,i,j,\ell} y),
\]
where $E_{\theta,i,j,\ell} $ and $F_{\theta,i,j,\ell}$ are sums of products of the matrices $R_{\theta}^{-1}$, $R_{\theta}$ and the first, second and third derivative matrices of $R_{\theta}$. Hence, from Condition~\ref{cond:cov:Y:theta} and Lemma~\ref{lem:Rtheta:m1:ij} used inside Lemma~\ref{lem:product:matrix}, we have $\sup_{\theta \in \Theta} | (E_{\theta,i})_{a,b} | \leq \Csup/(1+|s_a-s_b|^{d+\Cinf})$ and $\sup_{\theta \in \Theta} (F_{\theta,i})_{a,b} | \leq \Csup/(1+|s_a-s_b|^{d+\Cinf})$. Then, from GCT and Lemma 4 in \cite{furrer2016asymptotic}, we have $\sup_{\theta \in \Theta} \rho_1(E_{\theta,i}) \leq \Csup$ and $\sup_{\theta \in \Theta} \rho_1(F_{\theta,i}) \leq \Csup$. Hence, as in the proof of Theorem~\ref{theorem:consistency:thetaML}, we can show

\begin{equation} \label{eq:for:TCL:theta:4:ML}
    \sup_{\theta \in \Theta} \left| \frac{\partial L_{\theta}}{ \partial \theta_i \partial \theta_j \partial \theta_\ell } \right| = O_p(1).
\end{equation}

Also, $\lambda_1( M_{\theta_0} )$ is clearly bounded as $n \to \infty$. From Theorem~\ref{theorem:generic:TCL}, $\lambda_1( \Sigma_{\theta_0} )$ is bounded as $n \to \infty$. Hence, by considering subsequences along which $M_{\theta_0}$ and $\Sigma_{\theta_0}$ converge, and using \eqref{eq:for:TCL:theta:1:ML}, \eqref{eq:for:TCL:theta:2:ML}, \eqref{eq:for:TCL:theta:3:ML} and \eqref{eq:for:TCL:theta:4:ML}, we can proceed as in the proof of Proposition D.10 in \cite{bachoc14asymptotic} and show \eqref{eq:TCL:theta}.
\end{proof}

\begin{proof}[{\bfseries Proof of Theorem~\ref{theorem:consistency:thetaCV}}]
Let $\psi \in \mathcal{S}$ be fixed. We have
\[
    \Var( CV_{\psi} ) = \frac{1}{n} n \Var \left( \frac{1}{n} y^\top C_\psi^{-1}  \diag(C_\psi^{-1})^{-2} C_\psi^{-1} y \right).
\]

From Lemmas~\ref{lem:product:matrix}, \ref{lem:Rtheta:m1:ij} and \ref{lem:diag:Rtheta:m1:ij} (that can be trivially adapted by replacing $\theta$ by $\psi$), as well as
Theorem~\ref{theorem:generic:TCL}, applied with $A_n = C_\psi^{-1}  \diag(C_\psi^{-1})^{-2} C_\psi^{-1}$, we obtain $\Var( CV_{\psi} ) \to 0$ as $n \to \infty$.

For $i=1,\ldots,p-1$,
\[
    \frac{\partial CV_{\psi}}{ \partial \psi_{i}} = \frac{2}{n} y^\top A_{\psi,i} y \label{temp:derivCv}
\]
with
\[
    A_{\psi,i} = C_\psi^{-1} \diag(C_\psi^{-1})^{-2} \left( \diag \left( C_\psi^{-1} \frac{\partial C_{\psi}}{\partial \psi_{i}} C_\psi^{-1} \right) \diag(C_\psi^{-1})^{-1} - C_\psi^{-1} \frac{\partial C_{\psi}}{\partial \psi_{i}} \right) C_\psi^{-1}.\label{temp:P:CV}
\]

As in the proof of Theorem~\ref{theorem:consistency:thetaML}, GCT and Lemma 4 in \cite{furrer2016asymptotic} lead us to $\sup_{\psi \in \mathcal{S}} \lambda_1(A_{\psi,i}^\top A_{\psi,i}) \leq \Csup$, which in turn implies $\sup_{\psi \in \mathcal{S}} \rho_1(A_{\psi,i}) \leq \Csup$. It follows that

\begin{align*}
\max_{i=1,\ldots,p-1} \sup_{\psi \in \mathcal{S}}
\left|
\frac{\partial CV_{\psi}}{ \partial \psi_i}
\right|& \leq \sup_{\psi \in \mathcal{S}}
\left| \frac{2}{n} \left(y^\top A_{\psi,i} y \right) \right|\\
& \leq \Csup \frac{||y||^2}{n}\\
&= O_p(1).
\end{align*}

Hence, Theorem~\ref{theorem:consistency:thetaCV} can also be proved by proceeding as in the proof of Proposition 3.4 in \cite{bachoc14asymptotic}.
\end{proof}

\begin{proof}[{\bfseries Proof of Theorem~\ref{theorem:CLT:thetaCV}}]
From Condition~\ref{cond:cov:Y:theta} and Lemma~\ref{lem:Rtheta:m1:ij} used inside Lemma~\ref{lem:product:matrix}, we have for $i=1,\ldots,p-1$
\[
\frac{\partial CV_{\psi}}{ \partial \psi_i}  = \frac{2}{n} y^\top A_{\psi,i} y,
\]
where $A_{\psi,i}$ is a $n \times n$ matrix satisfying $\sup_{\psi \in \mathcal{S}}
| (A_{\psi,i})_{a,b} | \leq \Csup/(1+|s_a-s_b|^{d+\Cinf})$. As in the proof of Theorem~\ref{theorem:CLT:thetaML}, one can check that $\partial CV_{\psi_0} / \partial \psi_i $ has mean zero for $i=1,\ldots,p-1$.
Let $\partial CV_{\psi_0} / \partial \psi$ be the gradient column vector of $CV_{\psi}$ at $\psi_0$. From Theorem~\ref{theorem:generic:TCL}, with $\mathcal{L}_{\Gamma,\psi_0,n}$ the distribution of $\sqrt{n} \partial CV_{\psi_0} / \partial \psi$, as $n \to \infty$,
\begin{equation} \label{eq:for:TCL:theta:1:CV}
d_w ( \mathcal{L}_{\Gamma,\psi_0,n} , \cN[ 0 , \Gamma_{\psi_0} ] ) \to 0.
\end{equation}
In addition, for $i \in \{1 , \ldots , p-1\}$, the sequence $(n \Var(\partial CV_{\psi_0} / \partial \psi_i))$ is bounded and thus, the elements of $\Gamma_{\psi_0}$ are bounded too.

One can check that the mean value of $\partial CV_{\psi_0} / \partial \psi_i \partial \psi_j$ is $(N_{\psi_0})_{i,j}$. Furthermore, from Theorem~\ref{theorem:generic:TCL}, the variance of $\partial CV_{\psi_0} / \partial \psi_i \partial \psi_j$ goes to zero as $n \to \infty$. Hence, as $n \to \infty$,
\begin{equation} \label{eq:for:TCL:theta:2:CV}
\partial CV_{\psi_0} / \partial \psi_i \partial \psi_j
\to^p 
(N_{\psi_0})_{i,j}.
\end{equation}

It can be shown, similarly as in the proof of Proposition 3.7 in \cite{bachoc14asymptotic} that
\begin{equation} \label{eq:for:TCL:theta:3:CV}
\liminf_{n \to \infty}
\lambda_{p-1} ( N_{\psi_0} ) >0.
\end{equation}

Hence, \eqref{eq:in:TCL:theta:lambda1:CV} follows.
On the other hand, for $i,j=1,\ldots,p-1$, we have
\[
\frac{\partial CV_{\psi}}{ \partial \psi_i 
\partial \psi_j }
 = \frac{1}{n} y^\top D_{\psi,i,j} y,
\]
where $D_{\psi,i,j} $ is computed as a sum of products of the matrices $C_{\psi}^{-1}$, $C_{\psi}$, the first and second derivative matrices of $C_{\psi}$ and the $\diag$ operator (see e.g., \cite{bachoc14asymptotic}). Hence, from Condition~\ref{cond:cov:Y:theta} and Lemma~\ref{lem:Rtheta:m1:ij} used inside Lemma~\ref{lem:product:matrix}, we have $\sup_{\psi \in \mathcal{S}} | (D_{\psi,i})_{a,b} | \leq \Csup/(1+|s_a-s_b|^{d+\Cinf})$.

Similarly, for $i,j,\ell \in \{1 , \ldots , p-1\}$, we have
\[
  \frac{ \partial CV_{\psi}}{ \partial \psi_i \partial \psi_j  \partial \psi_\ell}
  = \frac{1}{n} y^\top E_{\psi,i,j,\ell} y,
\]
where $E_{\psi,i,j,\ell} $ is a sum of products of the matrices $C_{\psi}^{-1}$, $C_{\psi}$, the first, second and third derivative matrices of $C_{\psi}$ and the $\diag$ operator. Hence, from Condition~\ref{cond:cov:Y:theta} and Lemma~\ref{lem:Rtheta:m1:ij} used inside Lemma~\ref{lem:product:matrix}, we have $\sup_{\psi \in \mathcal{S}} | (E_{\psi,i,j,\ell})_{a,b} | \leq \Csup/(1+|s_a-s_b|^{d+\Cinf})$. Then, from GCT and Lemma 4 in \cite{furrer2016asymptotic}, we have $\sup_{\psi \in \mathcal{S}} \rho_1(E_{\psi,i,j,\ell}) \leq \Csup$. Hence, as in the proof of Theorem~\ref{theorem:consistency:thetaML}, we can show, for $i,j,\ell \in \{ 1 , \ldots , p-1\}$,

\begin{equation} \label{eq:for:TCL:theta:4:CV}
    \sup_{\psi \in \mathcal{S}} \left| \frac{\partial CV_{\psi}}{ \partial \psi_i \partial \psi_j \partial \psi_\ell } \right| = O_p(1).
\end{equation}

Also, $\lambda_1( N_{\psi_0} )$ is clearly bounded as $n \to \infty$. From Theorem~\ref{theorem:generic:TCL}, $\lambda_1( \Gamma_{\psi_0} )$ is bounded as $n \to \infty$. Hence, by considering subsequences along which $N_{\psi_0}$ and $\Gamma_{\psi_0}$ converge, and using \eqref{eq:for:TCL:theta:1:CV}, \eqref{eq:for:TCL:theta:2:CV}, \eqref{eq:for:TCL:theta:3:CV} and \eqref{eq:for:TCL:theta:4:CV}, we can proceed as in the proof of Proposition D.10 in \cite{bachoc14asymptotic} and show \eqref{eq:TCL:theta:CV}.
\end{proof}

\begin{proof}[{\bfseries Proof of Lemma~\ref{lem:cond:CV:implies:ML}}]
We have, with $\theta = (\sigma^2 , \psi) \in \Theta$ and $\theta_0 = (\sigma_0^2 , \psi_0)$,
\begin{flalign} \label{eq:for:cond:CV:imply:cond:ML:one}
&
\frac{1}{n}
\sum_{i,j=1}^n
\left(
c_{Y,\psi}( s_i-s_j )
-
c_{Y,\psi_0}( s_i-s_j )
\right)^2
\notag
\\
& =
\frac{1}{n}
\sum_{i,j=1}^n
\left(
\frac{k_{Y,\theta}( s_i-s_j )}{k_{Y,\theta}(0 )}
-
\frac{k_{Y,\theta_0}( s_i-s_j )}{k_{Y,\theta_0}( 0 )}
\right)^2
& \notag \\
& \leq 
\frac{2}{n}
\sum_{i,j=1}^n
\left(
\frac{k_{Y,\theta}( s_i-s_j )}{k_{Y,\theta}(0 )}
-
\frac{k_{Y,\theta_0}( s_i-s_j )}{k_{Y,\theta}( 0 )}
\right)^2
+
\frac{2}{n}
\sum_{i,j=1}^n
\left(
\frac{k_{Y,\theta_0}( s_i-s_j )}{k_{Y,\theta}(0 )}
-
\frac{k_{Y,\theta_0}( s_i-s_j )}{k_{Y,\theta_0}( 0 )}
\right)^2
\notag
\\
& \leq 
\frac{\Csup }{n}
\sum_{i,j=1}^n
\left(
k_{Y,\theta}( s_i-s_j )
-
k_{Y,\theta_0}( s_i-s_j )
\right)^2
+
\left(
\frac{1}{k_{Y,\theta}(0 )}
-
\frac{1}{k_{Y,\theta_0}(0 )}
\right)^2
\Csup,
&
\end{flalign}
where the second $\Csup$ comes from Lemma~\ref{lem:Rtheta:Rtheta:m1} and from the classical control of the square Frobenius norm by $n$ times the largest square eigenvalue, for $n \times n$ symmetric matrices. 
If Condition \ref{cond:asymptotic:identifiability:theta:CV} holds, then for all $\alpha>0$,
\[
\liminf_{n \to \infty}
\inf_{ || \psi - \psi_0 || \geq \alpha }
\frac{1}{n}
\sum_{i,j=1}^n
\left(
c_{Y,\psi}(s_i - s_j)
-
c_{Y,\psi_0} (s_i - s_j)
\right)^2
> 0.
\]
Consider a sequence $\theta_{n} = (\sigma_n^2 , \psi_n) \in \Theta$ such that $||\theta_n - \theta_0|| \geq \alpha$. If we can extract a subsequence $n_m$ such that $ \liminf_{m \to \infty} ( \sigma_{n_m}^2 - \sigma_0)^2 >0$, then clearly 
\[
\liminf_{m \to \infty}
\frac{1}{n_m}
\sum_{i,j=1}^{n_m}
\left(
k_{Y,\theta_{n_m}}( s_i-s_j )
-
k_{Y,\theta_0}( s_i-s_j )
\right)^2
> 0
\]
by considering the diagonal terms in the above double sum. If we can not extract such a subsequence, then we can extract a subsequence $n_m$ such that $||\psi_{n_m} - \psi_0|| \geq \alpha/2$ and $\sigma^2_{n_m} \to \sigma_0^2$ as $m \to \infty$. Along this subsequence 
\[
\liminf_{m \to \infty}
\frac{1}{n_m}
\sum_{i,j=1}^{n_m}
\left(
k_{Y,\theta_{n_m}}( s_i-s_j )
-
k_{Y,\theta_0}( s_i-s_j )
\right)^2
> 0
\]
from \eqref{eq:for:cond:CV:imply:cond:ML:one}. Hence, Condition \ref{cond:asymptotic:identifiability:theta} holds.

Let us now assume that Condition~\ref{cond:asymptotic:identifiability:local:theta} does not hold. We have, with $\theta = (\sigma^2 , \psi) \in \Theta$, with $\theta_0 = (\sigma_0^2 , \psi_0)$ and with $(\beta_1 , \ldots , \beta_p) = (\beta_1,\alpha_1,\ldots,\alpha_{p-1})$, where $\beta_1 \in \IR$ is arbitrary,
\begin{flalign} \label{eq:for:cond:CV:imply:cond:ML:two}
&
\frac{1}{n}
\sum_{i,j=1}^n
\left(
\sum_{\ell=1}^{p-1}
\alpha_{\ell}
\frac{\partial}{\partial \psi_{\ell}}
c_{Y,\psi_0  }( s_i-s_j )
\right)^2
&
\notag
\\
&
=
\frac{1}{n}
\sum_{i,j=1}^n
\left(
\sum_{\ell=1}^p
\beta_{\ell}
\frac{\partial}{\partial \theta_{\ell}}
\left(
\frac{k_{Y,\theta_0}( s_i-s_j )}{k_{Y,\theta_0}(0 )}
\right)
\right)^2
& \notag \\
& = 
\frac{1}{n}
\sum_{i,j=1}^n
\left(
\sum_{\ell=1}^p
\beta_{\ell}
\frac{\frac{\partial}{\partial \theta_{\ell}}k_{Y,\theta_0}( s_i-s_j )}{k_{Y,\theta_0}(0 )}
-
\sum_{\ell=1}^p
\beta_{\ell}
\frac{
k_{Y,\theta_0}( s_i-s_j )
\frac{\partial}{\partial \theta_{\ell}} k_{Y,\theta_0}(0 ) 
}{k_{Y,\theta_0}(0 )^2}
\right)^2
&
\notag \\
&
\leq 
\frac{2}{n}
\sum_{i,j=1}^n
\left(
\sum_{\ell=1}^p
\beta_{\ell}
\frac{\frac{\partial}{\partial \theta_{\ell}}k_{Y,\theta_0}( s_i-s_j )}{k_{Y,\theta_0}(0 )}
\right)^2
+
\frac{2}{n}
\sum_{i,j=1}^n
\left(
\sum_{\ell=1}^p
\beta_{\ell}
\frac{
k_{Y,\theta_0}( s_i-s_j )
\frac{\partial}{\partial \theta_{\ell}} k_{Y,\theta_0}(0 ) 
}{k_{Y,\theta_0}(0 )^2}
\right)^2
\notag \\
& \leq 
\frac{\Csup}{n}
\sum_{i,j=1}^n
\left(
\sum_{\ell=1}^p
\beta_{\ell}
\frac{\partial}{\partial \theta_{\ell}}k_{Y,\theta_0}( s_i-s_j )
\right)^2
+
\Csup
\left(
\sum_{\ell=1}^p
\beta_{\ell}
\frac{\partial}{\partial \theta_{\ell}} k_{Y,\theta_0}(0 )
\right)^2,
&
\end{flalign}
where the second $\Csup$ comes from Lemma~\ref{lem:Rtheta:Rtheta:m1} and from the classical control of the square Frobenius norm by $n$ times the largest square eigenvalue, for $n \times n$ symmetric matrices. If Condition~\ref{cond:asymptotic:identifiability:local:theta} does not hold, 
 there exists $(\beta_1^\star,\ldots,\beta_{p}^\star) \neq (0, \ldots , 0)$ and a subsequence $n_m$ such that
 \[
 \frac{1}{n_m}
 \sum_{i,j=1}^{n_m}
\left(
\sum_{\ell=1}^p
\beta_{\ell}^\star
\frac{\partial}{\partial \theta_{\ell}}k_{Y,\theta_0}( s_i-s_j )
\right)^2
\to_{m \to \infty} 0
 \]
 and thus, considering the diagonal elements in the double sum above, 
 \[
\beta_1^\star 
=
\sum_{\ell=1}^p
\beta_{\ell}^\star
\frac{\partial}{\partial \theta_{\ell}} k_{Y,\theta_0}(0 )
= 0.
 \]
Hence, from \eqref{eq:for:cond:CV:imply:cond:ML:two}, letting $(\beta^\star_1 , \ldots , \beta^\star_p) = (0,\alpha^\star_1,\ldots,\alpha^\star_{p-1})$, we have
\[
\frac{1}{n_m}
\sum_{i,j=1}^{n_m}
\left(
\sum_{\ell=1}^{p-1}
\alpha_{\ell}^\star
\frac{\partial}{\partial \psi_{\ell}}
c_{Y,\psi_0  }( s_i-s_j )
\right)^2
\to 0
\]
and thus Condition~\ref{cond:asymptotic:identifiability:local:theta:CV} does not hold.
\end{proof}

\begin{proof}[{\bfseries Proof of Theorem~\ref{theorem:CLT:joint}}]
Let $\lambda$ and $\gamma$ be two column vectors in $\mathbb{R}^p$ and $\mathbb{R}^{p-1}$. Let also
\[
    W_n = \lambda^\top (\hat{\theta}_{\text{ML}} - \theta_0) + \gamma^\top (\hat{\psi}_\text{CV} - \psi_0),
\]
for $n \in \mathbb{N}$.

From the proofs of Theorems~\ref{theorem:CLT:thetaML} and~\ref{theorem:CLT:thetaCV} (see also the proof of Proposition D.10 in \cite{bachoc14asymptotic} that is referred to there), we know
\begin{equation*}
    \sqrt{n} (\hat{\theta}_{\text{ML}} - \theta_0) = \sqrt{n} M^{-1}_{\theta_0} \frac{\partial}{\partial \theta} L_{\theta_0} + o_p(1)
\end{equation*}
and
\begin{equation*}
    \sqrt{n} (\hat{\psi}_\text{CV} - \psi_0) = \sqrt{n} N^{-1}_{\psi_0} \frac{\partial}{\partial \psi} CV_{\psi_0} + o_p(1).
\end{equation*}

Also, from Condition~\ref{cond:cov:Y:theta} and Lemma~\ref{lem:Rtheta:m1:ij} used inside Lemma~\ref{lem:product:matrix}, we have for $i=1,\ldots,p$,
\[
    \partial L_{\theta_0} / \partial \theta_i  =  \frac{1}{n} (y^\top A_{\theta_0,i} y) + c_{\theta_0},
\]
where $c_{\theta_0} \in \mathbb{R}$ is deterministic
and, for $i =1 , \ldots,p-1$,
\[
    \partial CV_{\psi_0} / \partial \psi_i  = \frac{1}{n} (y^\top B_{\psi_0,i} y),
\]
where $A_{\theta,i}$ is a $n \times n$ symmetric matrix satisfying $\sup_{\theta \in \Theta}
| (A_{\theta,i})_{a,b} | \leq \Csup/(1+|s_a-s_b|^{d+\Cinf})$ and $B_{\psi,i}$ is a $n \times n$ matrix satisfying $\sup_{\psi \in \mathcal{S}} | (B_{\psi,i})_{a,b} | \leq \Csup/(1+|s_a-s_b|^{d+\Cinf})$. As in the proofs of Theorems~\ref{theorem:CLT:thetaML} and~\ref{theorem:CLT:thetaCV}, one can check that $\partial L_{\theta_0} / \partial \theta_i$ has mean zero for $i=1,\ldots,p$ and $\partial CV_{\psi_0} / \partial \psi_i$ has mean zero for $i=1,\ldots,p-1$. Thus, we can rewrite
\[
 W_n
 = J_n
 - 
 \IE[J_n]
 +o_p(n^{-1/2})
\]
with
\[
    J_n = \frac{1}{n} y^\top \left( \sum_{i=1}^p (\lambda^\top  M^{-1}_{\theta_0})_i\ A_{\theta_0, i} + \sum_{i=1}^{p-1} (\gamma^\top  N^{-1}_{\psi_0})_i\ B_{\psi_0, i} \right) y.
\]

As the vectors $\lambda$ and $\gamma$ as well as the matrices $M^{-1}_{\theta_0}$ and $N^{-1}_{\psi_0}$ are fixed, the bound

\[
    \left\lvert \left( \sum_{i=1}^p (\lambda^\top  M^{-1}_{\theta_0})_i\ A_{\theta_0, i} + \sum_{i=1}^{p-1} (\gamma^\top  N^{-1}_{\psi_0})_i\ B_{\psi_0, i} \right)_{k,l} \right\rvert \leq \frac{ \Csup }{ 1 + | s_k - s_l |^{d + \Cinf }}
\]
holds for all $k,l \in \{1, \dots, n\}$.

Let then $\mathcal{L}_{J,\theta_0,n}$ be the distribution of $n^{1/2} ( J_n - \IE[J_n])$. Then, from Theorem~\ref{theorem:generic:TCL}, as $n \to \infty$, 
\[
    d_w \left( \mathcal{L}_{J,\theta_0,n}, \cN \left[ 0, n \var(J_n) \right] \right) \to 0.
\]

The variance can be written as
\begin{align*}
    n \var(J_n) =&\ \sum_{i=1}^p \sum_{j=1}^p (\lambda^\top  M^{-1}_{\theta_0})_i\ (\lambda^\top  M^{-1}_{\theta_0})_j\ (\Sigma_{\theta_0})_{i,j}
    + \sum_{i=1}^{p-1} \sum_{j=1}^{p-1} (\gamma^\top  N^{-1}_{\psi_0})_i\ (\gamma^\top  N^{-1}_{\psi_0})_j\ (\Gamma_{\psi_0})_{i,j}\\
    &+ 2 \sum_{i=1}^p \sum_{j=1}^{p-1} (\lambda^\top  M^{-1}_{\theta_0})_i\ (\gamma^\top  N^{-1}_{\psi_0})_j\ (\Omega_{\theta_0})_{i,j}
\end{align*}
with
\[
    \Omega_{i,j} = \cov \left( \sqrt{n} \frac{\partial}{\partial \theta_i} L_{\theta_0}, \sqrt{n} \frac{\partial}{\partial \psi_j} CV_{\psi_0} \right).
\]
Hence, by applying product by blocks we get the matrix form expression
\[
n    \var(J_n) = \left( \lambda^\top, \gamma^\top \right) D^{-1}_{\theta_0} \Psi_{\theta_0} D^{-1}_{\theta_0} \left( \begin{array}{cc}
        \lambda \\
        \gamma \\
    \end{array} \right).
\]
We conclude the proof by applying the Wald Theorem.
\end{proof}

\section*{Acknowledgements}
RF acknowledges the support of the Swiss National Science
Foundation SNSF-175529. FB acknowledges the support of a PEPS from the French Centre national de la recherche scientifique. This research was partly undertaken within the RISCOPE ANR project.

\end{document}